\documentclass[11pt,reqno]{amsart}
\usepackage[margin=1in]{geometry} 
\usepackage{amsmath,amsxtra,amssymb,amsthm,amsfonts,mathtools}
\usepackage{mathrsfs}
\usepackage{multirow}
\usepackage{tikz,quantikz}
\usepackage{physics}
\usepackage{bbm}
\usepackage{subfigure}
\usepackage{comment}
\usepackage{appendix}
\usepackage{tcolorbox}
\usetikzlibrary{arrows, automata}
\usepackage{color}
\usepackage[colorlinks = true,
            linkcolor = blue,
            urlcolor  = blue,
            citecolor = blue,
            anchorcolor = blue]{hyperref}

\allowdisplaybreaks

\numberwithin{equation}{section}
\newtheorem{lemma}{Lemma}[section]
\newtheorem{prop}[lemma]{Proposition}
\newtheorem{theorem}[lemma]{Theorem}
\newtheorem{assumption}[lemma]{Assumption}

\newtheorem{rem}[lemma]{Remark}
\newtheorem{remark}[lemma]{Remark}

\newtheorem{example}[lemma]{Example}
\newtheorem{definition}[lemma]{Definition}

\newcommand{\re}{\begin{rem}\rm}
	\newcommand{\mar}{\end{rem}}

\DeclareMathOperator{\CLSI}{\text{CLSI}}
\newcommand{\qd}{\end{proof}\vspace{0.5ex}}
\newcommand{\mb}{\mathbb}
\newcommand{\wt}{\widetilde}
\newcommand{\mc}{\mathcal}
\newcommand{\adm}{\text{ad}}
\newcommand{\Mz}{{\mathbb M}}
\newcommand{\Lip}{\mathrm{Lip}}
\newcommand{\Cost}{\mathrm{Cost}}

\newcommand{\fix}{\mathrm{fix}}

\newcommand{\pl}{\hspace{.1cm}}

\newcommand{\kl}{\pl \le \pl}

\newcommand{\lel}{\pl = \pl}

\newcommand{\CScb}{\Cost_{\mc S}^{cb}}

\newcommand{\RelEnt}{D}

\newcommand{\HS}[2]{\left\langle #1,#2\right\rangle_{\mathrm{HS}}} 
\newcommand{\BKM}[2]{\left\langle #1,#2\right\rangle_{\mathrm{BKM},\sigma}} 

\newcommand{\pw}[1]{{\color{blue}[PW: #1]}}

\makeatletter
\DeclareRobustCommand\smileotimes{\mathbin{\mathpalette\smile@otimes\relax}}
\newcommand{\smile@otimes}[2]{%
  \vbox{
    \ialign{##\cr
      \hidewidth$\m@th#1{}_\smile$\kern-\scriptspace\hidewidth\cr
      \noalign{\nointerlineskip\kern-1pt}
      $\m@th#1\otimes$\cr
    }%
  }%
}
\makeatother

\makeatletter
\DeclareRobustCommand\frownotimes{\mathbin{\mathpalette\frown@otimes\relax}}
\newcommand{\frown@otimes}[2]{%
  \vbox{
    \ialign{##\cr
      \hidewidth$\m@th#1{}_\frown$\kern-\scriptspace\hidewidth\cr
      \noalign{\nointerlineskip\kern-1pt}
      $\m@th#1\otimes$\cr
    }%
  }%
}
\makeatother

\title{Transportation cost and Contraction coefficient for channels on von Neumann algebras} 
\author[R. Araiza]{Roy Araiza}
\address{Department of Mathematics, IQUIST, University of Illinois at Urbana-Champaign}
\email{raraiza@illinois.edu}
\author[M. Junge]{Marius Junge}
\address{Department of Mathematics, IQUIST, University of Illinois at Urbana-Champaign}
\email{mjunge@illinois.edu}
\author[P. Wu]{Peixue Wu}
\address{Department of Applied Mathematics, Institute for Quantum Computing, University of Waterloo}
\email{p33wu@uwaterloo.ca}

\begin{document}
\begin{abstract}
We present a noncommutative optimal transport framework for quantum channels acting on von Neumann algebras. Our central object is the Lipschitz cost measure, a transportation-inspired quantity that evaluates the minimal cost required to move between quantum states via a given channel. Accompanying this is the Lipschitz contraction coefficient, which captures how much the channel contracts the Wasserstein-type distance between states. We establish foundational properties of these quantities, including continuity, dual formulations, and behavior under composition and tensorization. Applications include recovery of several mathematical quantities including expected group word length and Carnot-Carath\'eodory distance, via transportation cost. Moreover, we show that if the Lipschitz contraction coefficient is strictly less than one, one can get entropy contraction and mixing time estimates for certain classes of non-symmetric channels. 

\end{abstract}

\maketitle
\tableofcontents
 
\section{Introduction}\label{section:intro}
\subsection{From operator Lipschitz geometry to optimal transport theory}
First introduced by Connes~\cite{connes1992metric}, spectral triples serve as a method to encode distance on noncommutative spaces. Indeed, a spectral triple $(\mc A, \mathcal{H}, D)$, consists of a Hilbert space $\mathcal{H}$, an involutive algebra $\mc A$ of operators on $\mathcal{H}$, and a self-adjoint Dirac operator $D$ on $\mathcal{H}$ such that 
\begin{align*}
    |||X|||_L := \|[D,X]\|_{op} < \infty, 
\end{align*}
for all $X \in \mc A$. 
Subsequently, Rieffel's semimal work \cite{rieffel1999metrics,rieffel2004compact} axiomatized Lipschitz semi-norms on algebras and order unit spaces: 
\begin{enumerate}
    \item[(0)] $  |||\,\cdot\,|||_{L}:\mc A\longrightarrow[0,\infty)$;
    \item[(1)] $|||\textbf{1}|||_{L}=0$;
    \item[(2)] $|||x|||_{L}=|||x^{*}|||_{L},\ x \in \mc A$.  
\end{enumerate}
Such a seminorm \emph{generates} a metric on the normal state space $S(\mc A)$ via the dual formula
\begin{equation}\label{def:distance seminorm}
  W_{L}(\rho,\sigma) = \sup\left\{
       |\rho(x)-\sigma(x)|:x=x^{*}\in\mc A,\ |||x|||_{L}\le1 \right\},
  \quad
  \rho,\sigma\in S(\mc A).
\end{equation}


In the classical setting, Monge \cite{monge1781memoire} formulated the optimal transport problem as finding a mapping $ T : \mc X \rightarrow \mc Y$, where $\mc X,\mc Y$ are finite sets, such that $T$ rearranges an initial distribution $ \mu $ to a target distribution $ \nu $ while minimizing the total transportation cost:
$$
\inf_T \sum_{x \in X} c(x, T(x)) \, d\mu(x),
$$
where $ c(x, y) $ is the cost of moving a unit mass from $ x $ to $ y $. However, this formulation requires the transport map $ T $ to be measurable and mass-preserving, which makes the problem non-linear and, in some cases, ill-posed. To address this issue, Kantorovich \cite{kantorovich1942translocation} introduced the concept of transport plans $ \gamma $, which are modelled as probability measures on $ \mc X \times \mc Y $ with marginals $ \mu $ and $ \nu $, allowing for mass to be split and distributed among multiple destinations. Then the reformulated problem becomes
$$
W_{1,c}(\mu,\nu) = \inf_{\gamma \in \Pi(\mu, \nu)} \sum_{x \in \mc X, y\in \mc Y} c(x, y) \gamma(x, y),
$$
where $ \Pi(\mu, \nu) $ is the set of all transport plans (couplings) between $ \mu $ and $ \nu $. Replacing the cost function $c$ by a distance $d$ on $\mc X$, we have the celebrated Kantorovitch duality formulation~\cite{kantorovich1942translocation}:
\begin{equation}\label{def:dual formualtion classical}
    W_{1,d}(\mu,\nu)=\sup \left\{ \left|\int_{\mc X} f(x) d\mu(x) - \int_{\mc X} f(x)d\nu(x)\right|: |f(x)-f(y)| \leq d(x, y),\ \forall x, y \in \mc X\right\},
\end{equation}

Note the striking resemblance between Equation~\eqref{def:distance seminorm} and Equation~\eqref{def:dual formualtion classical}: the operator Lipschitz framework \emph{contains} Kantorovich’s duality formulation in the commutative cases. Our approach exploits three distinct advantages formulating transport in terms of Lipschitz seminorms:
\begin{enumerate}
\item \emph{Direct noncommutative generality:} states on arbitrary von Neumann algebras are treated on equal footing with classical probability measures.
\item \emph{Operator-algebraic tools:} commutators, derivations and completely bounded norms replace geometric couplings, enabling powerful functional-analytic estimates.
\item \emph{Seamless classical limit:} setting the von Neumann algebra as $L_{\infty}(\mc X)$ and a suitable seminorm immediately recovers Wasserstein distances without additional work.
\end{enumerate}
These features make the Lipschitz approach an ideal starting point for studying channels, and their contraction coefficients and mixing phenomena—topics that are central to this paper.

\subsection{Transportation cost and contraction coefficient of channels}
In this work, inspired by \eqref{def:distance seminorm}, we provide a systematic method to quantify the \textit{cost} of quantum transport plans, modeled as quantum channels, and to analyze the \textit{contraction coefficient} of channels under this distance, referred to as the \textit{Lipschitz constant}.

Let $\mc N$ be a von Neumann algebra, and $|||\cdot|||_L$ be a Lipschitz seminorm defined on a weak-$*$ dense involutive algebra $\mc A \subseteq \mc N$ (see Section \ref{sec: basic cost} for more general definition allowing operator space structure).

For a quantum channel $\Phi: \mc N \to \mc N$, which is assumed to be normal, unital and completely positive, we define two measures:
\begin{itemize}
    \item The \textit{transportation cost} (or \textit{Lipschitz Cost}):
    \begin{equation}\label{def:cost intro}
        \Cost_{L}(\Phi) := \sup_{\substack{x=x^* \\ |||x|||_L \le 1}} \|\Phi(x) - x\|_{op},
    \end{equation}
    where $\|\cdot\|_{op}$ denotes the operator norm.
    \item The \textit{contraction coefficient} (or \textit{Lipschitz constant}):
    \begin{equation}\label{def:Lipschitz constant intro}
        \Lip_L(\Phi) := \sup_{\substack{x=x^* \\ |||x|||_L \le 1}} \|\Phi(x)\|_L.
    \end{equation}
\end{itemize}
To streamline the presentation, we only consider channels $\Phi$ such that $\Phi(\mc A) \subseteq \mc A$. Recall that every von Neumann algebra $\mc N$ has a unique predual $\mc N_*$ \cite{sakai2012}, which is the Banach space of all normal (weak-$*$ continuous) linear functionals on $\mc N$. We denote by $S(\mc N)$ the normal state space of $\mc N$, i.e., unital positive elements in $\mc N_*$. Via standard duality argument, for the predual channel $\Phi_* : \mc N_* \to \mc N_*$, we have $\Cost_{L}(\Phi)= \sup_{\rho \in S(\mc N)} W_L\big(\rho,\Phi_*(\rho)\big)$, which serves as a natural measure of cost, reflecting the minimal resources required to implement the quantum transformation $\Phi$. This idea was first explored in \cite{PMTL,LBKJL} and later in \cite{araiza2023} as a means to quantify the cost of quantum operations.

In contrast, $\Lip_L(\Phi)$ quantifies the maximal “amplification factor” of the Lipschitz seminorm under the channel $\Phi$. A large Lipschitz constant indicates that the channel is capable of generating or amplifying correlations, while a small Lipschitz constant suggests that the channel tends to mix the system, thereby promoting convergence to equilibrium.

\subsection{Overview of the main results}
\subsubsection{Fundamental properties}
A first natural question is whether \eqref{def:Lipschitz constant intro} can be represented as the contraction coefficient under the Wasserstein metric, which motivates its name.  In the classical (commutative) setting these two quantities coincide by Kantorovich duality. Whether equality still holds for general von~Neumann algebras remained open, although one inequality is already known~\cite[Proposition 4.4]{gaorouze24}.  

As a first step toward closing this gap, we show that equality \emph{does} hold whenever the underlying semi-norm is Lipschitz, that is, obeys unit degeneracy and self-adjoint invariance:
\begin{prop}\label{prop:duality}
  \label{thm:Lip=contr}
  Let $\Phi\!:\mc N \!\to\! \mc N$ be a normal, unital, completely positive map and let $\Phi_*:\mc N_*\!\to\! \mc N_*$ denote its predual. For any Lipschitz semi-norm with domain $\mc A$ such that  $\Phi(\mc A) \subseteq \mc A$, we have
  \[
      \sup_{x=x^{*}\in\mc A}\frac{|||\Phi(x)|||_{L}}{|||x|||_{L}}
      \;=\;
      \sup_{\substack{\rho,\sigma\in S(\mc N)\\\rho\neq\sigma}}
      \frac{W_{L}\bigl(\Phi_*(\rho),\Phi_*(\sigma)\bigr)}
           {W_{L}(\rho,\sigma)}.
  \]
\end{prop}

Motivated by~\cite{wu2005noncommutativemetricsmatrixstate}, we next consider \emph{matrix} Lipschitz structures.  A \textit{matrix Lipschitz semi-norm} is a family $\{|||\!\cdot\!|||_{L^{(n)}}\}_{n\ge 1}$ with
$|||\!\cdot\!|||_{L^{(n)}}:\mb M_{n}(\mc A)\to[0,\infty)$ satisfying
\begin{enumerate}
    \item \textit{Degeneracy on the unit}: $|||\textbf{1}|||_{L^{(n)}} = 0.$
    \item \textit{Self-adjoint invariance (Lipschitz isometric)}: $|||x|||_{L^{(n)}} = |||x^*|||_{L^{(n)}} \quad \forall\, x \in \mb M_n(\mc A)$.
    \item $|||x \oplus y|||_{L^{(n+m)}} = \max\{|||x|||_{L^{(n)}}, |||y|||_{L^{(m)}}\},\quad x\in \mb M_n(\mc A), y \in \mb M_m(\mc A), n,m\ge 1$.
    \item $|||axb|||_{L^{(m)}} \le \|a\|_{op}\cdot |||x|||_{L^{(n)}} \cdot \|b\|_{op},\quad \quad x\in \mb M_n(\mc A), a\in \mb M_{m\times n}, b\in \mb M_{n\times m}$.
\end{enumerate}
Writing $|||\!\cdot\!|||_{L}:=|||\!\cdot\!|||_{L^{(1)}}$, we set
\begin{align*}
  &\Cost_{L}^{cb}(\Phi):=
    \sup_{n\ge 1}\;
    \sup_{\substack{x=x^{*}\in\mb M_{n}(\mc A)\\|||x|||_{L^{(n)}}\le 1}}
    \bigl\|\Phi^{(n)}(x)-x\bigr\|_{op},
  \\
  &\Lip_{L}^{cb}(\Phi):=
    \sup_{n\ge 1}\;
    \sup_{\substack{x=x^{*}\in\mb M_{n}(\mc A)\\|||x|||_{L^{(n)}}\le 1}}
    |||\Phi^{(n)}(x)|||_{L^{(n)}}.
\end{align*}
These two channel quantities enjoy the following properties, see Section \ref{sec: basic cost} and \ref{sec:basic Lipschitz constant} for the formal statements and details.
\\
\renewcommand{\arraystretch}{1.2}
\begin{center}
\begin{tabular}{|p{0.42\textwidth}|p{0.25\textwidth}|p{0.25\textwidth}|}
\hline
\centering \textbf{Properties} &
\centering \textbf{Transportation cost} &
\centering \textbf{Lipschitz constant}\tabularnewline
\hline\hline
\centering Perfect channel &
\centering $0$ &
\centering $1$\tabularnewline
\hline
\centering Composition of channels &
\centering Subadditivity &
\centering Submultiplicativity\tabularnewline
\hline
\centering Convex combination of channels &
\centering Convexity &
\centering Convexity\tabularnewline
\hline
\centering (Commutator semi-norm) Universal upper bound &
\centering $\kappa(\mc S)\,\|\Phi-id\|_{cb}$ &
\centering $2\,\kappa(\mc S)\,\displaystyle\sup_{s\in\mc S}\|s\|_{op}$\tabularnewline
\hline
\centering (Commutator semi-norm) Tensorization &
\centering Tensor additivity &
\centering Tensor maximality\tabularnewline
\hline
\end{tabular}
\end{center}
Here $\mc S$ is the operator resource set and $\kappa(\mc S)$ is defined as the \textit{expected length}. See Section \ref{sec: basic cost}, in particular Definition~\ref{def: expected length} for the definition. 
\subsubsection{Applications}
We now illustrate the versatility of our framework through several concrete applications.  We begin by showing that, in the purely classical setting, the transportation cost of a conditional expectation coincides with the familiar \emph{expected word length} in a finite group.
\medskip 

\noindent \textbf{Word length via conditional expectations (see Section~\ref{sec:finite group})}: 

Let $G$ be a finite group equipped with a finite symmetric generating set $S$. The \emph{word length} of $g\in G$ (with respect to~$S$) is the minimal number of generators or their inverses required to express~$g$, that is
\begin{equation}
    \ell_{S}(g) = \min \{ n \in \mathbb{N} : g = s_1 s_2 \cdots s_n, \; s_i \in S \}.
\end{equation}
Endow~$G$ with the normalized Haar measure $\mu$, so that $\mu(\{g\})=|G|^{-1}$ for all $g\in G$. Consider
\[
  \mathcal N \;=\; L_{\infty}(G,\mu)
  \;=\;\bigl\{f:G\to\mb C \,\big|\, f\text{ bounded}\bigr\},
  \quad
  \mathcal H \;=\; L_{2}(G,\mu),
\]
with inner product $ \langle f_{1},f_{2}\rangle_{\mathcal H}
  \;=\;
  \int_G f_{1}(g)\,\overline{f_{2}(g)}\,d\mu(g)$. We define a conditional expectation
\begin{equation}
    E_{\fix}(f) = \left(\int_G f(g) \, d\mu(g)\right) \textbf{1}_G,\quad f\in \mc N.
\end{equation}

With this in hand, we can now characterize the expected group word length via the Lipschitz cost measure:
\begin{theorem}[c.f. Theorem~\ref{thm:expected_length}] \label{thm:word-length}
   There exists a Lipschitz semi-norm $|||\cdot|||_L$ on $\mc N$ such that 
    \[
        \Cost_{L}(E_{\fix})= \int_G \ell_{S}(g)\, d\mu(g).
    \]
\end{theorem}

\noindent \textbf{Carnot-Carath\'eodory distance for locally compact Lie groups (see Section \ref{sec:geometric}):}

Let $G$ now be a locally compact Lie group with Lie algebra~$\mathfrak g$ and Haar measure~$\mu$.  For a unitary representation $\pi:G\longrightarrow U(\mathcal K)$, we denote the induced representation of the Lie algebra $\mathfrak{g}$ on the same Hilbert space $d\pi : \mathfrak{g} \to \mathfrak{u}(\mathcal{K})$, where $\mathfrak{u}(\mathcal{K})$ denotes the skew-adjoint operators on $\mathcal{K}$. 

For a fixed subspace $\mathfrak{h} \subseteq\mathfrak g$, we obtain a non-commutative Lipschitz semi-norm on $\mathcal B(\mathcal K)$ as follows: let $X_{1},\dots,X_{m}$ be an orthonormal basis of $\mathfrak{h}$.  For
  \[
    \mc A
    \;:=\;
    \bigl\{x\in\mathcal B(\mathcal K)\bigm|[d\pi(X_{j}),x]\in\mathcal B(\mathcal K)\text{ for all }j\bigr\},
  \]
  define the Lipschitz semi-norm:
  \[
    |||x|||_L
    := 
    \Bigl\|\sum_{j=1}^{m}[d\pi(X_{j}),x]^{*}[d\pi(X_{j}),x]\Bigr\|^{1/2}_{op},
    \qquad x\in \mc A.
  \]

\begin{theorem}[c.f. Theorem \ref{main:CC distance lower bound}]\label{thm:CC distance}
  For every unitary representation $\pi$ and $g\in G$,
  \[
    \Cost_{L}\bigl(\adm_{\pi(g)}\bigr)\;\le\;d_{\mathfrak{h}}(g,e),\ \adm_{\pi(g)}(x):= \pi(g)^* x\,\pi(g).
  \]
  where $d_{\mathfrak{h}}$ is the sub-Riemannian (Carnot–Carath\'eodory) distance on~$G$ associated with~$\mathfrak{h}$ defined in \eqref{def:CC distance}.
\end{theorem}

\medskip
\noindent \textbf{Entropy contraction and mixing times (see Section~\ref{sec:application Lipschitz constant})}: 

We next use the Lipschitz constant formalism to bound entropy contraction coefficients and derive mixing-time estimates for (not necessarily symmetric) quantum channels on $\mc N = \mb M_d$. Since we are in finite dimension, there exists a trace and we adopt the usual quantum convention, i.e., $\Phi$ is trace preserving and completely positive. Assume $\Phi$ has a unique fixed state $\sigma$ with full rank, define the entropy contraction coefficient
\begin{equation}
    \eta^D(\Phi,\sigma):= \sup_{\rho \in S(\mc N)} \frac{D(\Phi(\rho) \| \sigma)}{D(\rho \| \sigma)} \in [0,1].
\end{equation}
Motivated by \cite{caputo2025entropy}, a general upper bound on the entropy contraction coefficient is given as follows:
\begin{theorem}[c.f. Theorem \ref{main: entropy contraction upper}] \label{thm:entropy-contraction}
  Assume additionally $\Phi$ is a strictly positive channel, then for any Lipschitz semi-norm $|||\cdot|||_{L}$,
  \[
    \eta^{D}(\Phi,\sigma) \le \Lip_L(\Phi^*)
    \max\left \{
      \Lip_L(\Phi^{\mathrm{*,BKM}}),\mathrm{LogLip}(\Phi,\sigma)
    \right\},
  \]
  where $ \Phi^{\mathrm{*,BKM}}:=  \Gamma_\sigma^{-1} \circ \Phi \circ \Gamma_\sigma,\quad \Gamma_\sigma(x):= \int_0^1 \sigma^s x\, \sigma^{1-s}ds$ and the logarithmic Lipschitz constant is defined by 
  \begin{align*}
       \mathrm{LogLip}_L(\Phi,\sigma)
    \;:=\;
    \sup_{\rho\in S(\mathcal N),\,\rho>0}
      \frac{|||\,\log\Phi(\rho)-\log\sigma\,|||_{L}}
           {|||\,\log\rho-\log\sigma\,|||_{L}}.
  \end{align*}
\end{theorem}
Note that in the commutative case, the Lipschitz constant of $\Phi^{\mathrm{*,BKM}}$ and logarithmic Lipschitz constant for $\Phi$ can be upper bounded by $1$ if the channel satisfies non-negative sectional curvature condition~\cite{caputo2025entropy}. However, since in the noncommutative setting there is no coupling notion, more technical issues arise.

Nevertheless, using bounded mean oscillation (BMO) techniques developed in \cite{caspers2020bmo}, we can bound $\mathrm{LogLip}(\Phi,\sigma)$ for general unital channels, yielding the following mixing-time estimate:
\begin{theorem}[c.f. Theorem \ref{main: mixing time upper}]\label{thm:mixing-time}
  Let $\Phi$ be unital with unique fixed point $\mathbf1/d$.  There exists a universal constant $c_{\mathrm{abs}}>0$ such that, for any commutator semi-norm $|||\cdot|||_{L}$ with $\Lip_{L}(\Phi)<1$ and $\Lip_{L}(\Phi^{*})\le 1$,
  \[
   t_{\mathrm{mix}}\!\bigl(\varepsilon,\Phi\bigr)
    \;\le\;
    c_{\mathrm{abs}}\,
    \frac{\log(1/\varepsilon)+\log d}
         {-\log\bigl(\Lip_{L}(\Phi)\bigr)\;-\;\log\!\bigl(\Lip_{L}(\Phi^{*})\bigr)}.
  \]
\end{theorem}
Finally, we relate the two quantitative measures introduced above and obtain a cost induced mixing time bound that underpins the lower-bound results of~\cite{ding2024lower} for simulating open quantum systems. Recall that for a quantum channel $\Phi$ with fixed point algebra $\mc N_\fix$ and a conditional expectation $E_\fix$ onto $\mc N_\fix$, define the cost induced mixing time
\begin{align*}
     t_{\mathrm{mix}}^L(\varepsilon,\Phi^*)= \inf\{n \ge 1: \Cost_L^{cb}((\Phi^*)^n)\ge (1-\varepsilon) \Cost_L^{cb}(E_\fix)\}. 
\end{align*}
The following bound for cost induced mixing time connects the Lipschitz constant and transportation cost:
\begin{prop}[c.f. Proposition \ref{prop:cost induced mixing time upper}]\label{prop:transport-mix}
  Suppose $\Phi^{*}\circ E_{\fix}=E_{\fix}\circ\Phi^{*}=E_{\fix}$ and $\Lip_{L}^{\mathrm{cb}}(\Phi)<1$.  Then, for every $\varepsilon>0$,
  \[
    t_{\mathrm{mix}}^{L}\!\bigl(\varepsilon,\Phi^{*}\bigr)
    \;\le\;
    \frac{\log(1/\varepsilon)}
         {-\log\bigl(\Lip_{L}^{\mathrm{cb}}(\Phi^{*})\bigr)}.
  \]
\end{prop}
\subsection{Relation to other works}
A first step toward transportation cost for channels was taken in~\cite{PMTL,LBKJL}, where the quantum Wasserstein–1 distance was defined in $n$-qubit system. Yet, as noted in~\cite{araiza2023}, the resulting metric is upper bounded by~$n$; it therefore fails to resolve genuinely non–local behaviour in large systems. Our framework overcomes this limitation by allowing \emph{arbitrary} Lipschitz semi-norms: by selecting suitable commutator norms one can tailor the cost to the physical locality structure of interest.

The notion of a Lipschitz constant for quantum channels was introduced in \cite{gaorouze24}, and several ad-hoc estimates were obtained later \cite{rouzé2024optimalquantumalgorithmgibbs}. The present paper advances the subject in two directions:

\begin{enumerate}
\item We give a unified, operator-algebraic treatment of \emph{both} transportation cost and Lipschitz constant for arbitrary von~Neumann algebras, clarifying their functorial and tensorial properties.
\item We provide the first systematic analysis of entropy contraction for non-symmetric primitive channels and translate these bounds into mixing-time estimates.
\end{enumerate}
One interesting direction for future research is to study the transportation cost and contraction coefficient of channels via the coupling approach to noncommutative optimal transport theory, which is currently an active area of investigation, see for example~\cite{biane2000freeprobabilityanaloguewasserstein,D_Andrea_2010, agredo13,caglioti2021optimaltransportquantumdensities, De_Palma_2021, Beatty_2025, anshu2025quantumwassersteindistancesquantum}.

\medskip
\noindent\textbf{Organization of the paper.}
Section~\ref{sec:prelim} recalls the basic operator-algebraic tools used throughout.  
Section~\ref{sec:basic} establishes the fundamental properties of transportation cost and Lipschitz constant in full generality.  
Section~\ref{sec:application cost} illustrates the cost formalism by recovering several familiar quantitative invariants.  
Finally, Section~\ref{sec:application Lipschitz constant} applies our Lipschitz estimates to entropy contraction coefficients and mixing-time bounds.

\section{Preliminary} \label{sec:prelim}
\textbf{Basics on von Neumann algebras}.--- General references for von Neumann algebras are \cite{sakai2012, dixmier2011neumann}.\\ 
Let $\mathcal{H}$ be a Hilbert space, and let $\mathcal{B}(\mathcal{H})$ denote the algebra of bounded linear operators on $\mathcal{H}$. A subset $\mathcal{N} \subseteq \mathcal{B}(\mathcal{H})$ is called a \textit{von Neumann algebra} if it is a unital $*$-subalgebra that is closed in the weak operator topology (or equivalently, in the strong operator topology). By the double commutant theorem, this is equivalent to $\mathcal{N} = \mathcal{N}''$, where $\mathcal{N}':= \{x \in \mc B(\mc H) : [a,x] = ax - xa = 0,\ \forall a\in \mc N\}$ denotes the commutant of $\mathcal{N}$ in $\mathcal{B}(\mathcal{H})$.

The weak operator topology (also known as weak$^*$-topology) on $\mathcal{B}(\mathcal{H})$ is the topology of pointwise convergence on $\mathcal{H}$; that is, a net $(x_\alpha) \subseteq \mathcal{B}(\mathcal{H})$ converges to $x \in \mathcal{B}(\mathcal{H})$ in the weak operator topology if and only if
$$
\langle \xi, x_\alpha \eta \rangle \to \langle \xi, x \eta \rangle \quad \text{for all } \xi, \eta \in \mathcal{H}.
$$
For any operator $x \in \mathcal{N}$, the \textit{operator norm} is defined by
$$
\|x\|_{\mc N} = \| x \|_{op} = \sup_{\xi \in \mc H, \| \xi \| = 1} \| x \xi \| = \sqrt{ \sup \sigma(x^* x) },
$$
where $\sigma(x^* x)$ denotes the spectrum of $x^* x$. We denote the sets of self-adjoint and positive elements in $\mathcal{N}$ as
$$
\mathcal{N}_{s.a.} = \{ x \in \mathcal{N} : x = x^* \}, \quad \mathcal{N}_{+} = \{ x^*x: x \in \mathcal{N}\}.
$$
We denote by $\mathbf{1}_{\mathcal{N}}$ the identity operator in $\mathcal{N}$.

Every von Neumann algebra $\mathcal{N}$ has a \textit{predual} $\mathcal{N}_*$, which is the Banach space of all weak$^*$-continuous linear functionals on $\mathcal{N}$. A linear functional $\phi : \mathcal{N} \to \mathbb{C}$ is called a \textit{state} if it is positive (i.e., $\phi(x^* x) \geq 0$ for all $x \in \mathcal{N}$) and normalized (i.e., $\phi(\mathbf{1}_{\mathcal{N}}) = 1$). We denote by $\mathcal{D}(\mathcal{N})$ the set of normal states on $\mathcal{N}$ (i.e., weak$^*$-continuous states). A state $\phi$ is called \textit{faithful} if $\phi(x^* x) = 0$ implies $x = 0$. If $\mathcal{H}$ is separable, then $\mathcal{N}$ admits at least one faithful normal state.

A state $\tau$ is called a \textit{tracial state} if it additionally satisfies the \textit{trace property}:
$$
\tau(x y) = \tau(y x), \quad \forall x, y \in \mathcal{N}.
$$
Not every von Neumann algebra admits a tracial state. In particular, type III von Neumann algebras (ever-present in quantum field theory) do not admit any tracial states. For type III von Neumann algebras, every normal positive linear functional $\phi$ has a non-trivial modular group $\{ \sigma_t^\phi \}$ obtained via Tomita–Takesaki theory \cite{sakai2012}. This modular group cannot be trivialized, and thus there is no cyclic invariance under the trace property because the modular automorphism group disrupts this cyclicity.

When we require the trace property, we assume that $\mathcal{N}$ is a \textit{semifinite von Neumann algebra}. That is, $\mathcal{N}$ admits a faithful normal semifinite trace $\tau : \mathcal{N}_{+} \to [0, \infty]$. More precisely, $\tau$ satisfies:
\begin{enumerate}
    \item[(i)] Faithfulness: If $\tau(x^* x) = 0$, then $x = 0$.
    \item[(ii)] Normality: For any bounded increasing net $\{ x_\alpha \} \subset \mathcal{N}_{+}$,
    $$
    \tau\left( \sup_\alpha x_\alpha \right) = \sup_\alpha \tau(x_\alpha).
    $$
    \item[(iii)] Semifiniteness: For any nonzero $a \in \mathcal{N}_{+}$, there exists a nonzero $e \in \mathcal{N}_+$ with $e \leq a$ such that $\tau(e) < \infty$.
\end{enumerate}

For example, $\mathcal{B}(\mathcal{H})$, where $\mathcal{H}$ is separable and infinite-dimensional, is a semifinite von Neumann algebra when equipped with the usual (unbounded) trace $\operatorname{Tr}$. A von Neumann algebra $\mathcal{N}$ is called \textit{finite} if the family formed of the finite, i.e., $\tau(\textbf{1}_{\mathcal N}) < \infty$, normal traces separates points of $\mathcal N$. If $\mathcal N$ admits a single faithful normal finite trace, then it is finite. Though, in general a finite von Neumann algebra may fail to have a faithful finite trace. All finite-dimensional von Neumann algebras are examples of finite von Neumann algebras.

In the prescence of a trace $\tau$, one may naturally consider the corresponding \emph{noncommutative $L_p$ spaces}. Define $\mathcal{N}_0 = \bigcup_{\tau(e) < \infty} e \mathcal{N} e$, where the union is over all projections $e \in \mathcal{N}$ with $\tau(e) < \infty$. For $1 \leq p < \infty$, the noncommutative \( L^p \) space \( L_p(\mathcal{N}, \tau) \) is defined as the completion of \( \mathcal{N}_0 \) with respect to the norm
$$
\| a \|_{L_p(\mathcal{N}, \tau)} = \left( \tau( |a|^p ) \right)^{1/p}.
$$
We will often use the shorthand notation \( \| \cdot \|_p \) when no confusion arises. We identify \( L_\infty(\mathcal{N}) := \mathcal{N} \), and the predual space \( \mathcal{N}_* \cong L_1(\mathcal{N}, \tau) \) via the duality
\begin{equation}\label{eqn:tracial state}
    a\in L_1(\mc N) \longleftrightarrow \phi_a\in \mc N_*,\quad \phi_a(x)=\tau(ax).
\end{equation}
This identification allows us to treat elements of \( L_1(\mathcal{N}) \) as normal linear functionals on \( \mathcal{N} \).

\textbf{Channels.}---A channel \( \Phi: \mathcal{N} \to \mathcal{N} \) is a normal, unital, and completely positive linear map on a von Neumann algebra \( \mathcal{N} \). Let $\mathfrak{C}(\mathcal N)$ denote the set of all channels on $\mc N$. Here, completely positive means for any matrix algebra $\mb M_n$, the linear operator $id_{\mb M_n} \otimes \Phi$ is positive. Here $id$ denotes a identity channel in this paper.

We adopt the \textit{Heisenberg picture}, where channels act on observables within the algebra \( \mathcal{N} \). This contrasts with the more common \textit{Schr\"odinger picture} in quantum information theory, where channels act on states.

The predual map \( \Phi_*: \mathcal{N}_* \to \mathcal{N}_* \), representing the Schr\"odinger picture, is defined via the duality between \( \mathcal{N} \) and its predual \( \mathcal{N}_* \). Specifically, for all \( \rho \in \mathcal{N}_* \) and \( x \in \mathcal{N} \), we have
\begin{equation}
    \Phi_*(\rho)(x) = \rho(\Phi(x)).
\end{equation}
This relation ensures that the action of \( \Phi_* \) on states corresponds to the action of \( \Phi \) on observables. In fact, when \( \mathcal{N} \) is semifinite and equipped with a faithful normal semifinite trace \( \tau \), we can relate \( \Phi \) and \( \Phi_* \) via \eqref{eqn:tracial state}:
\begin{equation}
    \tau\big( \rho\, \Phi(x) \big) = \tau\big( \Phi_*(\rho)\, x \big).
\end{equation}
By setting \( x = \mathbf{1}_{\mathcal{N}} \), it follows that \( \Phi_* \) is \textit{trace-preserving}:
\begin{equation}
    \tau\big( \Phi_*(\rho) \big) = \tau(\rho).
\end{equation}
Thus, the predual map \( \Phi_* \) preserves the trace, aligning with the definition of a quantum channel in the Schr\"odinger picture. 

Note that every completely positive map between von Neumann algebras is automatically a complete contraction:
\begin{equation}\label{ineqn:contractive}
    \|\Phi\|_{cb} \le 1,\quad \|\Phi\|_{cb}:= \sup\{ \|(id_{\mb M_n} \otimes \Phi)(X)\|_{op}: \|X\|_{op} \le 1,\ X \in \mb M_n\otimes \mc N, n\ge 1\}.
\end{equation} In the case that $\Phi: \mc N \to \mc N$ is unital completely positive, then \[\|\Phi: \mc N \to \mc N\|_{cb} = \|\Phi: \mc N \to \mc N\|_{op} = \|\Phi(\textbf{1}_{\mc N})\| = 1.\]

\textbf{Conditional expectations}.--- If $\mc N_{\fix} \subseteq \mc N$ is a von Neumann subalgebra, then a conditional expectation $E_{\fix}: \mc N \to \mc N_{\fix}$ is a channel such that 
\begin{equation}
    E_{\fix}(n_1 x n_2) = n_1 E_{\fix}(x) n_2,\quad \forall x \in \mc N,\ n_i \in \mc N_{\fix}.
\end{equation}
The index $\mc I(E_{\fix})$ of a conditional expectation $E_{\fix}$ is defined by 
\begin{equation}\label{def:index}
    \mc I(E_{\fix}) = \inf\{c>0:\ x^* x \le  c E_{\fix}(x^*x),\ \forall x \in \mc N\},
\end{equation}
and we denote $\mc I^{cb}(E_{\fix}) = \sup_{n\ge 1} \mc I(id_{\mb M_n} \otimes E_{\fix})$.

We illustrate this notion in finite dimension. Recall that any finite dimensional von Neumann algebra is isomorphic to 
\begin{equation}\label{eqn:finite dim algebra}
\bigoplus_{i=1}^n\mc B(\mc H_i)\otimes \mb C\cdot\textbf{1}_{\mc K_i}, \mc H = \bigoplus_{i=1}^n \mc H_i \otimes \mc K_i,\quad d_{\mc H_i} = \text{dim} \mc H_i,\ d_{\mc K_i} = \text{dim} \mc K_i
\end{equation}
Assume $\mc N_{\fix}$ has the form \eqref{eqn:finite dim algebra} and $\mc N = \mc B(\mc H)$. Denote $P_i$ as the projection of $\mc H$ onto $\mc H_i\otimes \mc K_i$, then there exists a family of states $\tau_i \in \mc D(\mc K_i)$ such that 
\begin{align}\label{eqn:form of conditional expectation}
& E_{\fix}(X) = \bigoplus_{i=1}^n \tr_{\mc K_i}\big(P_iXP_i(1_{\mc H_i}\otimes \tau_i) \big) \otimes 1_{\mc K_i}, \quad E_{\fix*}(\rho) = \bigoplus_{i=1}^n \tr_{\mc K_i}\big(P_i \rho P_i\big) \otimes \tau_i,
\end{align}
where $\tr_{\mc K_i}$ is the partial trace with respect to $\mc K_i$. It is also easy to see that a state $\sigma$ is invariant for $E_{\fix}$, i.e., $E_{\fix*}(\sigma) = \sigma$ if and only if $
\sigma = \bigoplus_{i=1}^n p_i \sigma_i\otimes \tau_i,$ for some states $\sigma_i \in \mc D(\mc H_i)$ and probability distribution $\{p_i\}_{i=1}^n$. The conditional expectation is completely determined if an invariant state is specified, and we denote it as $E_{\fix,\sigma}$. We drop the index $\sigma$ if it is clear from the context. Note that the invariant state is not unique and the flexibility comes from different choices of $\sigma_i$ and the probability distribution $\{p_i\}_{i=1}^n$. The explicit calculation of the index of tracial conditional expectation is given in \cite{GJL20entropy}:
\begin{align}
\mc I(E_{\fix,\tr}) = \sum_{i=1}^n\min\{d_{\mc H_i}, d_{\mc K_i}\}d_{\mc K_i},\quad \mc I^{cb}(E_{\fix,\tr}) =\sum_{i=1}^n d_{\mc K_i}^2.
\end{align}

\section{Basic properties of transportation cost and Lipschitz constant of channels}\label{sec:basic}
Let $\mc N$ be a von Neumann algebra, and $\mc A \subseteq \mc N$ is a weak-$*$ dense subalgebra. A \textit{matrix Lipschitz semi-norm} introduced in \cite{wu2005noncommutativemetricsmatrixstate} is a family of seminorms $|||\cdot|||_{L^{(n)}}: \mb M_n(\mc A) \to [0,\infty),\ n\ge 1$ such that 
\begin{enumerate}
    \item \label{assump:unit vanish} \textit{Degeneracy on the Unit}: $|||\mathbf{1}_{\mb M_n(\mc A)}|||_{L^{(n)}} = 0.$
    \item \label{assump:self adjoint} \textit{Self-Adjoint Invariance (Lipschitz isometric)}: $|||x|||_{L^{(n)}} = |||x^*|||_{L^{(n)}} \quad \forall\, x \in \mb M_n(\mc A)$.
    \item \label{assump:max direct sum} $|||x \oplus y|||_{L^{(n+m)}} = \max\{|||x|||_{L^{(n)}}, |||y|||_{L^{(m)}}\},\quad x\in \mb M_n(\mc A), y \in \mb M_m(\mc A), n,m\ge 1$.
    \item \label{assump: triangle inequality} $ |||axb|||_{L^{(m)}} \le \|a\|_{op}\cdot |||x|||_{L^{(n)}} \cdot \|b\|_{op},\quad \quad x\in \mb M_n(\mc A), a\in \mb M_{m\times n}, b\in \mb M_{n\times m}$.
\end{enumerate}
Here $\|\cdot\|_{op}$ denotes the operator norm. We denote $|||\cdot|||_L = |||\cdot|||_{L^{(1)}}$.

For a quantum channel $\Phi: \mc N \to \mc N$, which is assumed to be unital, normal, and completely positive, denote 
\begin{equation}
   id_{\mb M_n}\otimes \Phi = \Phi^{(n)} 
\end{equation}
as the $n$-th amplification of $\Phi$. We define two quantities for channels. As mentioned in the introduction, we only consider channels $\Phi(\mc A)\subseteq \mc A$.
\begin{definition}\label{def:cost and constant}
The transportation cost (or Lipschitz Cost) of a quantum channel $\Phi$ and its complete transportation cost are defined by 
\begin{align}
    & \Cost_{L}(\Phi) := \sup_{\substack{x=x^* \\ |||x|||_L \le 1}} \|\Phi(x) - x\|_{op}, \\
    & \Cost_{L}^{cb}(\Phi) := \sup_{n\ge 1} \sup \left\{\|\Phi^{(n)}(x) - x\|_{op}:\ x=x^*\in \mb M_n(\mc A), |||x|||_{L^{(n)}} \le 1\right\}.
\end{align}
The Lipschitz constant of a quantum channel $\Phi$ and its complete Lipschitz constant are defined by 
\begin{align}
    & \Lip_L(\Phi) := \sup_{\substack{x=x^* \\ |||x|||_L \le 1}} \|\Phi(x)\|_L, \label{def: Lipschitz constant}\\
    & \Lip^{cb}_L(\Phi) := \sup_{n\ge 1}\sup \left\{\|\Phi^{(n)}(x)\|_{L^{(n)}}:\ x=x^*\in \mb M_n(\mc A), |||x|||_{L^{(n)}} \le 1\right\}. \label{def: complete Lipschitz constant}
\end{align}
\end{definition}
When the semi-norm is defined on the full von Neumann algebra, i.e., $\mc A = \mc N$, we are in the usual quantum setting. Since $\Phi(\mc A) \subseteq \mc A$, without loss of generality we can actually assume the semi-norm is defined on $\mc N$, via standard approximation procedure. Recall that every von Neumann algebra $\mathcal N$ possesses a unique predual
$\mathcal N_{*}$~\cite{sakai2012}, the Banach space of all normal
(weak$^{*}$‑continuous) linear functionals on~$\mathcal N$. Denote by $S(\mathcal N)$ its (normal) state space. 
To discuss the dual semi-norm, we make use of the identification
\begin{equation}\label{eq:complete-isometry}
   \mb M_n(\mathcal N_{*})
=   \operatorname{CB}(\mathcal N,\,\mb M_n) = \{f: \mc N \to \mb M_n: f\text{\ completely\ bounded}\}.
\end{equation}
We can actually define the dual seminorm on $\mb M_n(\mc A^*) = \mathrm{CB}(\mc A^{**},\mb M_n)$ in which $\mb M_n(\mc N_*)$ embeds into. In particular, the operator space dual structure on $\mc N_*$ is inherited via the inclusion $\mc N_* \subseteq \mc A^*$.

One may define the seminorm $|||\cdot|||_{L^{(n)}}^*$  for $f \in \mb M_n(\mc N_*)$ as follows:
\begin{equation}\label{eqn:dual-seminorm}
    |||f|||_{L^{(n)}}^* := \sup\left\{ \left\|\langle\!\!\!\langle f,x \rangle\!\!\!\rangle \right\|_{op} : x \in \mb M_m(\mc N),\, |||x|||_{L^{(m)}}\le 1,\ m\in \mb N \right\}.
\end{equation}
Via Smith lemma \cite{Paulsen2003}, the supremum over $m$ can be choose as $m \le n$. Here $\langle\!\!\!\langle f,x \rangle\!\!\!\rangle$ denotes the matrix pairing 
\begin{equation}
    \langle\!\!\!\langle f,x \rangle\!\!\!\rangle := \left[f_{ij}(x_{{kl}})\right] \in \mb M_{mn},\ f = [f_{ij}] \in \mb M_{n}(\mc N_*),\ x= [x_{kl}] \in \mb M_m(\mc N).
\end{equation}
It is well known (see, e.g., \cite{rieffel1999metrics}) that if $|||\cdot|||_{L^{(n)}}$ is self-adjoint invariant, then the supremum in
\eqref{eqn:dual-seminorm} can be restricted to self‑adjoint elements:
for every $f=f^{*}\in \mb M_n(\mathcal N_{*})$,
\begin{equation}\label{eqn:self adjoint dual}
  |||f|||_{L^{(n)}}^* := \sup\left\{ \left\|\langle\!\!\!\langle f,x \rangle\!\!\!\rangle \right\|_{op} : x = x^*\in \mb M_m(\mc N),\, |||x|||_{L^{(m)}}\le 1,\ m\in \mb N \right\}.
\end{equation}
Using a standard off-diagonal trick in operator space theory, the supremum over $m$ can be given by $m \le 2n$. Accordingly, one can defined an induced metric on the space $\operatorname{CP}(\mathcal N,\,\mb M_n) = \{\rho: \mc N \to \mb M_n: \rho\ \text{completely\ positive}\}$: 
\begin{equation}\label{def:distance selfadjoint}
    W_{L^{(n)}}(\rho,\sigma) =|||\rho - \sigma|||_{L^{(n)}}^*,\ \rho,\sigma \in \operatorname{CP}(\mathcal N,\,\mb M_n). 
\end{equation}
Let us call $W_{L}$ the \emph{Wasserstein L-metric}. In \cite{wu2005noncommutativemetricsmatrixstate}, the author further restricts to unital, completely positive maps. However, as we will see below, to establish a duality relation, it is important to stay within completely positive maps since the unital property is too restrictive.

It is straightforward to show that
\begin{equation}
    \Cost_{L}(\Phi) = \sup_{\substack{x=x^* \\ |||x|||_L \le 1}} \|\Phi(x) - x\|_{op} = \sup_{\rho \in S(\mc N)} W_L\big(\rho,\Phi_*(\rho)\big),
\end{equation}
In fact, note that $x = x^*$, $\|\Phi(x) - x\|_{op} = \sup_{\rho \in S(\mc N)} \rho(\Phi(x) - x)$. Taking the supremum and exchanging the order
\begin{align*}
    \sup_{\substack{x=x^* \\ |||x|||_L \le 1}} \|\Phi(x) - x\|_{op}&  = \sup_{\substack{x=x^* \\ |||x|||_L \le 1}} \sup_{\rho \in S(\mc N)} \rho(\Phi(x) - x) \\
    & = \sup_{\rho \in S(\mc N)} \sup_{\substack{x=x^* \\ |||x|||_L \le 1}} (\Phi_*(\rho) - \rho )(x) \\
    & = \sup_{\rho \in S(\mc N)} \|\Phi_*(\rho) - \rho\|_L^*.
\end{align*} 
\begin{remark}
Given such a matrix Lipschitz seminorm $\{L^{(n)}\}_{n\ge 1}$, if one wishes to realize $\Cost^{cb}_L(\Phi)$ as a Wasserstein L-metric, then one must consider alternate pairing for $\mb M_n(\mc N)$. In particular, one must consider the matricial state space $$S(\mb M_n(\mc N)): = \{F: \mb M_n(\mc N) \to \mb C: F\,\,\text{unital completely positive}\}$$ rather than $\mb M_n(\mc N_*)$, and use the natural scalar pairing given as \[\langle \mb M_n(\mc N), \mb M_n(\mc N)_*\rangle,\quad \langle x,F\rangle = F(x).\] 
\end{remark}


A slightly more involved duality argument shows that under the assumptions of degeneracy on the unit and self-adjoint invariance, the (complete) Lipschitz constant coincides with the (complete) contraction coefficient with respect to the Wasserstein L-metric defined in \eqref{def:distance selfadjoint}:
\begin{prop}
    Under the assumptions \eqref{assump:unit vanish}--\,\eqref{assump: triangle inequality} and $\Phi(\mc A) \subseteq \mc A$, we have
    \begin{equation}\label{eqn:lip-contract}
        \Lip_L(\Phi) = \sup_{\rho\neq\sigma \in S(\mc N)} \frac{W_L(\Phi_*(\rho),\Phi_*(\sigma))}{W_L(\rho,\sigma)},\quad
                \Lip_L^{cb}(\Phi)= \sup_{n \in \mb N} \sup_{\substack{\rho,\sigma \in \mathrm{CP}(\mc N, \mb M_n)\\ \rho(\mathrm{\textbf{1}}) = \sigma(\mathrm{\textbf{1}})}} \frac{W_{L^{(n)}}(\Phi^{(n)}_*(\rho),\Phi^{(n)}_*(\sigma))}{W_{L^{(n)}}(\rho,\sigma)}.
    \end{equation} 
\end{prop}
\begin{proof}
We present the proof for $\Lip_L(\Phi)$; the case of $\Lip^{cb}_L(\Phi)$ follows analogously, requiring only standard operator-space refinements. Moreover, since $\Phi(\mc A) \subseteq \mc A$, we can assume $\mc A = \mc N$ to streamline the proof.
Since 
\[
||| \mathbf{1}_{\mc N} |||_L = 0,
\]
the triangle inequality implies that for any $\lambda\in \mathbb{C}$ and any $x\in \mc N$, we have
\[
|||x + \lambda \, \mathbf{1}_{\mc N} |||_L = |||x|||_L.
\]
Thus, $|||\cdot|||_L$ descends to a well-defined seminorm on the quotient $\mc N/(\mathbb{C}\,\mathbf{1}_{\mc N})$. In particular, for any $x\in \mc N$,
\[
|||x|||_L = \sup\Big\{\, |f(x)| : f\in \mc N_*,\ f(\mathbf{1}_{\mc N}) = 0,\ |||f|||_L^* \le 1 \,\Big\}.
\]
Because $\Phi$ is unital, an analogous representation holds for $\Phi(x)$:
\begin{equation}\label{eqn:dual-phi}
||| \Phi(x) |||_L = \sup\Big\{\, |f(\Phi(x))| : f\in \mc N_*,\ f(\mathbf{1}_{\mc N}) = 0,\ |||f|||_L^* \le 1 \,\Big\}.
\end{equation}

We now claim that if $x = x^*$, then in \eqref{eqn:dual-phi} one may restrict the supremum to self-adjoint functionals. To see this, let $\varepsilon>0$ and choose $f_0\in \mc N_*$ with $f_0(\mathbf{1}_{\mc N})=0$ and $|||f_0|||_L^*\le 1$ such that
\[
|f_0(\Phi(x))| \ge |||\Phi(x)|||_L - \varepsilon.
\]
By multiplying $f_0$ by a suitable phase factor $e^{i\theta}$, we may assume without loss of generality that $f_0(\Phi(x))\in \mathbb{R}$. Then, setting
\[
f := \frac{f_0 + f_0^*}{2},
\]
we obtain a self-adjoint functional satisfying $f(\mathbf{1}_{\mc N}) = 0$, $|||f|||_L^*\le 1$, and 
\[
f(\Phi(x)) = f_0(\Phi(x)).
\]
Thus, \eqref{eqn:dual-phi} may be rewritten for self-adjoint $x$ as
\[
||| \Phi(x) |||_L = \sup_{\substack{f=f^* \in \mc N_* \\ f(\mathbf{1}_{\mc N}) = 0,\; |||f|||_L^* \le 1}} |f(\Phi(x))|.
\]
Taking the supremum over all self-adjoint $x$ with $|||x|||_L\le 1$, we have
\[
\sup_{\substack{x=x^* \\ |||x|||_L \le 1}} |||\Phi(x)|||_L = \sup_{\substack{x=x^* \\ |||x|||_L \le 1}} \sup_{\substack{f=f^* \in \mc N_* \\ f(\mathbf{1}_{\mc N}) = 0,\; |||f|||_L^* \le 1}} |f(\Phi(x))|.
\]

Exchanging the order of supremum yields
\[
\sup_{\substack{x=x^* \\ |||x|||_L \le 1}} |||\Phi(x)|||_L = \sup_{\substack{f=f^* \in \mc N_* \\ f(\mathbf{1}_{\mc N}) = 0,\; |||f|||_L^* \le 1}} \sup_{\substack{x=x^* \\ |||x|||_L \le 1}} |f(\Phi(x))| = \sup_{\substack{f=f^* \in \mc N_* \\ f(\mathbf{1}_{\mc N}) = 0,\; |||f|||_L^* \le 1}} |||\Phi_*(f)|||_L^*,
\]
where for the last equality above, we used \eqref{eqn:self adjoint dual} under the assumption \eqref{assump:self adjoint}.

Finally, note that every self-adjoint functional $f\in \mc N_*$ with $f(\mathbf{1}_{\mc N}) = 0$ can be expressed (up to scaling) as a difference of states; that is, there exist $\rho,\sigma\in S(\mc N)$ and $c\in \mathbb{R}$ such that $f = c(\rho-\sigma)$. 

Hence, one may identify
\[
\sup_{\substack{f=f^* \in \mc N_* \\ f(\mathbf{1}_{\mc N}) = 0,\; |||f|||_L^* \le 1}} |||\Phi_*(f)|||_L^* = \sup_{\rho,\sigma \in S(\mc N),\; |||\rho-\sigma|||_L^* \le 1} |||\Phi_*(\rho-\sigma)|||_L^*.
\]
By the definition of the Wasserstein distance $W_L$ in \eqref{def:distance selfadjoint}, the last expression is equal to
\[
\sup_{\rho\neq \sigma \in S(\mc N)} \frac{W_L\big(\Phi_*(\rho),\Phi_*(\sigma)\big)}{W_L(\rho,\sigma)}.
\]
This completes the proof of \eqref{eqn:lip-contract}.
\end{proof}

\subsection{Basic properties of transportation cost}\label{sec: basic cost}
Our principle goal is to quantify the cost of implementing certain quantum processes modeled as channels on von Neumann algebras. 
Mathematically, our goal is to find an appropriate cost functional $C$ defined on channels $$C: \mathfrak{C}(\mc N)\to \mb R_{\ge 0},$$
such that it satisfies the following axioms:
\begin{enumerate}
\item (\textit{No cost}) $C(id) = 0$.
\item (\textit{Subadditivity under concatenation}) $C(\Phi \circ \Psi) \le C(\Phi)+C(\Psi)$. 
\item (\textit{Convexity}) For any probability distribution $\{p_i\}_{i\in I}$ and any channels $\{\Phi_i\}_{i \in I}$, we have 
\begin{equation}
C(\sum_{i\in I} p_i\Phi_i) \le \sum_{i\in I} p_iC(\Phi_i).
\end{equation} 
\item (\textit{Universal upper bound}) $C(\Phi)$ is upper bounded by a universal constant depending on the underlying dimension.
\item (\textit{Tensor additivity}) For finitely many channels $\{\Phi_i\}_{1\le i \le m}$, we have 
\begin{equation}
C(\bigotimes_{i=1}^m \Phi_i) = \sum_{i=1}^m C(\Phi_i).
\end{equation}
\end{enumerate}	
We now verify that $\mathrm{Cost}^{cb}_L(\Phi)$ satisfies the above axioms $(1)-(3)$. Under some additional assumptions, $(4),(5)$ also hold. We begin by showing that $\mathrm{Cost}_L(\Phi), \mathrm{Cost}^{cb}_L(\Phi)$ defined in Definition \ref{def:cost and constant} satisfy the first three properties:
\begin{prop}\label{main:elementary property}
The (complete) Lipschitz cost satisfies
\begin{enumerate}
\item \label{cost: no cost} (\textit{No cost}) $\Cost_{L}(id) = \Cost^{cb}_{L}(id) = 0$.
\item \label{cost: subadditivity}(\textit{Subadditivity under concatenation}) For two channels $\Phi,\Psi: \mc N \to \mc N$, we have 
\begin{equation}
    \Cost_{L}(\Phi \circ \Psi) \le \Cost_{L}(\Phi) + \Cost_{L}(\Psi),\quad  \Cost^{cb}_{L}(\Phi \circ \Psi) \le \Cost^{cb}_{L}(\Phi) + \Cost^{cb}_{L}(\Psi).
\end{equation}
\item \label{cost: convexity} (\textit{Convexity}) For any $p_1,p_2 \in [0,1]$ with $p_1 + p_2 = 1$, we have 
\begin{equation}
 \Cost_{L}(p_1 \Phi + p_2 \Psi) \le p_1  \Cost_{L}(\Phi) + p_2  \Cost_{L}(\Psi).
\end{equation} 
\end{enumerate}
\end{prop}
\begin{proof}
Property \eqref{cost: no cost} follows from the definition. To show \eqref{cost: subadditivity}, for any $n\ge 1$, for any $x=x^*\in \mb M_n(\mc A), |||x|||_{L^{(n)}} \le 1$,
\begin{align*}
    \|(\Phi \circ \Psi)^{(n)}(x) - x\|_{op} & = \|(\Phi \circ \Psi)^{(n)}(x) - \Phi^{(n)}(x) + \Phi^{(n)}(x)- x\|_{op} \\
    & \le \|\Phi^{(n)}\circ (\Psi^{(n)}(x) - x)\|_{op} + \|\Phi^{(n)}(x)- x\|_{op} \\
    & \le \|\Psi^{(n)}(x) - x\|_{op} + \|\Phi^{(n)}(x)- x\|_{op},
\end{align*}
where we use the triangle inequality and $\|\Phi^{(n)}\|_{op} = 1$ for channels. Taking the supremum over $x$ and $n\ge 1$, we conclude the proof of \eqref{cost: subadditivity}. Property \eqref{cost: convexity} follows from the triangle inequality. 
\end{proof}

To show the remaining two properties, we require additional assumptions. Suppose $\mc S \subseteq \mc B(\mc H)$\footnote{Note that in some cases, $\mc S$ can be a set of unbounded operators(for example differential operators) we postpone this discussion for future work. } is a finite set of operators and we call $\mc S$ the \textit{resource set}. In this paper, we assume $\mc S$ satisfies the following properties:
\begin{assumption}\label{assumption:resource}
The commutant of $\mc S$, denoted as $\mc S'$, is a von Neumann subalgebra. Moreover, there exists a normal conditional expectation $E_{\mc S' \cap \mc N}: \mc N \to \mc S' \cap \mc N$. 
\end{assumption}
We recall the existence of such a conditional expectation is equivalent to $\mc S' \cap \mc N$ is invariant under the modular automorphism
group of some faithful normal state on $\mc N$ \cite{dixmier2011neumann}. Note that the conditional expectation is not unique but it is fixed throughout this paper.
We consider the two families of matrix
Lipschitz seminorms ($\ell_{\infty}$ and $\ell_2$ type):
\begin{equation}\label{def: commutator seminorm}
    |||x|||_{\mc S^{(n)}} : = \sup_{s \in \mc S} \left\|[\mb I_n\otimes s,x]\right\|_{op}, \quad |||x|||_{\mc S^{(n)},2} : = \left\| \left(\sum_{s \in \mc S} [\mb I_n\otimes s,x]^* [\mb I_n\otimes s,x] \right)^{\frac{1}{2}}\right\|_{op}
\end{equation}
for $ x \in \mb M_n(\mc N)$. The associated (complete) Lipschitz costs are denoted
\begin{equation}
    \Cost^{(cb)}_{\mc S}(\Phi),\quad \Cost^{(cb)}_{\mc S,2}(\Phi),
\end{equation}
and we focus mainly on the $\ell_{\infty}$‑type $\Cost^{(cb)}_{\mathcal S}(\Phi)$ for now, and we discuss $\ell_{2}$ analogue in Section \ref{sec:geometric}. Because elements of the commutant have zero seminorm, the cost can
blow up:
\begin{lemma}
    Suppose $E_{\mc S'\cap \mc N}$ is a conditional expectation onto $\mc S'\cap \mc N$, then if $\Phi \in \mathfrak{C}(\mc N)$ such that $\Phi \circ E_{\mc S'\cap \mc N} \neq E_{\mc S'\cap \mc N}$, then $\Cost_{\mc S}(\Phi) = \infty$.
\end{lemma}
\begin{proof}
    First, it is easy to see that $x \in \mc S'\cap \mc N$ if and only if $|||x|||_S = 0$. Since $\Phi \circ E_{\mc S'\cap \mc N} \neq E_{\mc S'\cap \mc N}$, one can choose $x \neq 0$, such that $\|\Phi(E_{\mc S'\cap \mc N}(x)) - E_{\mc S'\cap \mc N}(x)\|_{op} = c >0$. Note that for any $k>0$, we have $|||k \cdot E_{\mc S'\cap \mc N}(x)|||_S = k|||E_{\mc S'\cap \mc N}(x)|||_S = 0$ and $\|\Phi(k E_{\mc S'\cap \mc N}(x)) - k E_{\mc S'\cap \mc N}(x)\|_{op} = kc$. Since $k>0$ is arbitrary, we take the supremum and we get $\Cost_{\mc S}(\Phi) = \infty$. 
\end{proof}
Therefore, in this paper we restrict attention to channels that fix $\mc S' \cap \mc N$, i.e., $\Phi \circ E_{\mc S'\cap \mc N} = E_{\mc S'\cap \mc N}$. When $\mc S'$ is given by scalars, this condition is automatic
because every unital channel fixes the scalars. Based on this observation, we denote 
\begin{equation}\label{def: fixed point algebra}
    \mc N_{\fix} :=  \mc S'\cap \mc N,\quad E_{\fix}:= E_{\mc S'\cap \mc N}.
\end{equation}
Now we can present the \textit{universal upper bound} for Lipschitz cost:
\begin{prop}[Universal upper bound] \label{main:universal bound}
    For a channel $\Phi$ satisfying $\Phi \circ E_{\fix} = E_\fix$, we have 
    \begin{equation}
        \Cost_{\mc S}(\Phi) \le  \Cost_{\mc S}(E_{\fix})\, \|\Phi - id\|_{op},\quad \Cost_{\mc S}^{cb}(\Phi) \le  \Cost^{cb}_{\mc S}(E_{\fix})\, \|\Phi - id\|_{cb}.
    \end{equation}
\end{prop}
\begin{proof}
    For any $n\ge 1$ and $x = x^* \in \mb M_n(\mc N)$, we have 
    \begin{align*}
        \|\Phi^{(n)}(x)- x\|_{op}& = \|\Phi^{(n)}(x)- E_{\fix}^{(n)}(x) + E_{\fix}^{(n)}(x) - x\|_{op} \\
        & = \|\Phi^{(n)}(x)- \Phi^{(n)}\circ E_{\fix}^{(n)}(x) + E_{\fix}^{(n)} (x) - x\|_{op} \\
        & = \|(\Phi^{(n)} - id) \circ (E_{\fix}^{(n)} - id)(x)\|_{op} \\
        & \le \|(\Phi^{(n)} - id)\|_{op} \|(E_{\fix}^{(n)} - id)(x)\|_{op}
    \end{align*}
    Taking the supremum over $x$ and $n\ge 1$, we conclude the proof.
\end{proof}
\begin{definition}\label{def: expected length}
      The constant $\Cost^{cb}_{\mc S}(E_{\fix})=: \kappa(\mc S)$ is called the \textit{expected length} of $\mc S$.
\end{definition}
To establish \textit{tensor additivity}, we first specify how to define Lipschitz cost of the tensor product channel $\Phi_1 \otimes \Phi_2$, where $\Phi_i: \mc N_i\to \mc N_i,\ i = 1,2$, and $\mc N_i \subseteq \mc B(\mc H_i)$. Given $\mc S_1 \subset \mc B(\mc H_1), \mc S_2 \subset \mc B(\mc H_2)$, define the resource set for the composite system by
   \begin{equation}
   \mc S_1 \vee \mc S_2:= (\mc S_1 \otimes \textbf{1}_{\mc B(\mc H_2)})\cup (\textbf{1}_{\mc B(\mc H_1)} \otimes \mc S_2) = \{s_1\otimes \textbf{1}_{\mc B(\mc H_2)}, \textbf{1}_{\mc B(\mc H_1)}\otimes s_2: s_1\in \mc S_1, s_2\in \mc S_2\}. 
   \end{equation}
Here $\textbf{1}_{\mc B(\mc H)}$ is the identity operator acting on $\mc H$. In general, for $d$-composite system, suppose we have resources for each single system $\mc S_i \subset \mc B(\mc H_i), 1\le i \le d$, the resource set for the composite system is given by 
   \begin{equation}\label{def: composite resource}
    \wt{\mc S_j}:= \{\widetilde{X_j}\big|X \in \mc S_j\}, \quad \vee_{j=1}^d \mc S_j:= \bigcup_{j=1}^d \wt{\mc S_j}= \{\widetilde{X_j}\big|X \in \mc S_j, 1\le j\le d\},
   \end{equation}
where $$\widetilde{X_j} = \textbf{1}_{\mc B(\mc H_1)}\otimes \cdots \otimes \textbf{1}_{\mc B(\mc H_{j-1})} \otimes\underbrace{X}_{\text{j-th\ position}}  \otimes\, \textbf{1}_{\mc B(\mc H_{j+1})} \otimes \cdots \otimes \textbf{1}_{\mc B(\mc H_d)}.$$
As done in \eqref{def: commutator seminorm}, we may define the matrix Lipschitz semi-norm on $\bigotimes_{i=1}^d \mc N_i$, which is the weak-$*$ closure of the algebraic tensor product. For each channel $\Phi_j$, a natural induced channel $\wt{\Phi_j}$ on $\bigotimes_{i=1}^d \mc N_i$ is defined by 
\begin{equation}\label{def: embedded channel}
    \wt{\Phi_j} = id_{\mc N_1}\otimes \cdots \otimes id_{\mc N_{j-1}} \otimes\underbrace{\Phi_j}_{j-th\ position}  \otimes\, id_{\mc N_{j+1}} \otimes \cdots \otimes id_{\mc N_d}
\end{equation}

We are able to confirm tensor additivity axiom for composite-system Lipschitz cost: 
\begin{prop}[Tensor additivity]\label{main:additivity}
Suppose $\Phi_j: \mc N_j \to \mc N_j$ are quantum channels and $\mc S_j \subset \mc B(\mc H_j)$ are the resource sets for each single system for any $1\le j \le d$, then 
\begin{equation}
\begin{aligned}
    \Cost_{\vee_{j=1}^d \mc S_j}^{cb}(\Phi_1 \otimes \cdots \otimes \Phi_d) \le \sum_{j=1}^d \Cost_{\wt {\mc S_j}}^{cb}(\wt {\Phi_j}).
\end{aligned}
\end{equation}
Moreover, if
\begin{equation}\label{eqn:lifted cost}
    \Cost_{\wt {\mc S_j}}^{cb}(\wt {\Phi_j}) = \Cost_{\mc S_j}^{cb}(\Phi_j),
\end{equation}
we have tensor additivity.
\end{prop}
\begin{proof}
We only need to show the case when $d=2$ because the general case follows from induction. 

\noindent \emph{(i) Subadditivity.} For any $x \in \mb M_n(\mc N_1 \otimes \mc N_2)$ and $n \ge 1$, using the decomposition 
\begin{equation}
    \Phi_1 \otimes \Phi_2 - id = (\Phi_1 \otimes id)(id \otimes \Phi_2 - id) + (\Phi_1 \otimes id - id)
\end{equation}
and $\Phi_1 \otimes id_{\mc N_2}$ is unital completely positive thus $\|\Phi_1 \otimes id_{\mc N_2}\|_{op} = 1$, we have 
\begin{align*}
    \|(\Phi_1 \otimes \Phi_2)^{(n)}(x) - x\|_{op} \le \|(id_{\mc N_1} \otimes \Phi_2)^{(n)}(x) - x\|_{op} + \|(\Phi_1 \otimes id_{\mc N_2})^{(n)}(x) - x\|_{op},
\end{align*}
Therefore, note that $|||x|||_{\wt {\mc S_i}^{(n)}} \le |||x|||_{(\mc S_1 \vee \mc S_2)^{(n)}},\ i = 1,2$, taking the supremum over $n\ge 1$, we get
\begin{equation}
    \Cost_{\mc S_1 \vee \mc S_2}^{cb}(\Phi_1 \otimes \Phi_2) \le \Cost_{\wt {\mc S_1}}^{cb}(\wt {\Phi_1}) + \Cost_{\wt {\mc S_2}}^{cb}(\wt {\Phi_2}).
\end{equation}
\emph{(ii) Superadditivity.} By definition and the assumption that $\Cost_{\wt {\mc S_i}}^{cb}(\wt {\Phi_i}) = \Cost_{\mc S_i}^{cb}(\Phi_i)$, for any $\varepsilon\in (0,1)$, there exist $n_1,n_2 \in \mb N$, and $f_i = f_i^* \in \mb M_{n_i}(\mc N_i)$, $i=1,2$, such that $|||f_i|||_{\mc S_i^{(n_i)}}\le 1$, and
\begin{equation}\label{superadditivity:key step 1}
(1-\varepsilon)\Cost_{\mc S_i}^{cb}(\Phi_i)\le \|\Phi^{(n_i)}(f_i) - f_i\|_{op}.
\end{equation}
Since $\Phi^{(n_i)}(f_i) - f_i$ is self-adjoint, there exist states $\rho_i \in S(\mb M_{n_i}(\mc N_i))$ with $\rho_i(\mb I_{n_i} \otimes \textbf{1}_{\mc N_i}) = 1$,
\begin{align}\label{superadditivity:key step 2}
    \rho_i\left(\Phi^{(n_i)}(f_i) - f_i\right)= \|\Phi^{(n_i)}(f_i) - f_i\|_{op}.
\end{align}
Now we define 
\begin{equation}
F = f_1 \otimes \mb I_{n_2}\otimes \textbf{1}_{\mc N_2} + \mb I_{n_1} \otimes \textbf{1}_{\mc N_1} \otimes f_2 \in \mb M_{n_1}(\mc N_1) \otimes \mb M_{n_2}(\mc N_2),
\end{equation}
we have
\begin{align*}
 &|||F|||_{(\mc S_1 \vee \mc S_2)^{(n_1n_2)}} \\
 & = \max\{\sup_{s_1\in \mc S_1} \|[\mb I_{n_1}\otimes s_1\otimes \mb I_{n_2} \otimes  \textbf{1}_{\mc N_2}, F]\|_{op}, \sup_{s_2\in \mc S_2} \|[\mb I_{n_1}\otimes \textbf{1}_{\mc N_1}\otimes \mb I_{n_2} \otimes s_2, F]\|_{op}\} \\
 & = \max\{|||f_1|||_{\mc S_1^{(n_1)}}, |||f_2|||_{\mc S_2^{(n_2)}} \}  \le 1. 
\end{align*}
Finally,  
\begin{align*}
& \Cost_{\mc S_1 \vee \mc S_2}^{cb}(\Phi_1 \otimes \Phi_2) \ge \|(\Phi^{(n_1)} \otimes \Phi^{(n_2)})(F) - F\|_{op} \\
& \ge \bigg| (\rho_1 \otimes \rho_2) \bigg( \big(  \Phi^{(n_1)}(f_1) - f_1 \big) \otimes  \mb I_{n_2}\otimes \textbf{1}_{\mc N_2} +  \mb I_{n_1}\otimes \textbf{1}_{\mc N_1} \otimes \big( \Phi^{(n_2)}(f_2) - f_2\big) \bigg) \bigg| \\
& = \|\Phi^{(n_1)}(f_1) - f_1 \|_{op} + \|\Phi^{(n_2)}(f_2) - f_2\|_{op} \ge (1-\varepsilon)(\Cost_{S_1}^{cb}(\Phi_1)+ \Cost^{cb}_{S_2}(\Phi_2)).
\end{align*}
For the last equality, we used \eqref{superadditivity:key step 2} and for the last inequality, we used \eqref{superadditivity:key step 1}. Since $\varepsilon$ is arbitrary, we have $\Cost_{\mc S_1 \vee \mc S_2}^{cb}(\Phi_1 \otimes \Phi_2)\ge \Cost_{S_1}^{cb}(\Phi_1)+ \Cost^{cb}_{S_2}(\Phi_2)$ which concludes the proof of superadditivity.
\end{proof}

\begin{remark}
    A sufficient condition such that \eqref{eqn:lifted cost} holds is that all the von Neumann algebras are approximately finite dimensional:
\begin{equation}
  \mc N_i = \overline{\bigcup_{\alpha \in \Lambda} \mc M_\alpha^i}^{\mathrm{SOT}},\quad \mathrm{dim}\mc M_\alpha^i < \infty,
\end{equation}
where SOT stands for strong operator topology and $\Lambda$ is a net, directed by inclusion.
\end{remark}
\subsection{Basic properties of Lipschitz constant}\label{sec:basic Lipschitz constant}
The \emph{Lipschitz constant} of a channel quantifies how strongly the
channel can \emph{expand} or \emph{contract} a prescribed Lipschitz seminorm. It should have the following physical meaning:
\begin{itemize}
\item \emph{Contraction}, $\Lip(\Phi)<1$: 
      $\Phi$ dissipates fine structures encoded by the Lipschitz semi-norm; the dynamics
      is mixing.
\item \emph{Isometry}, \(\Lip(\Phi)=1\): 
      $\Phi$ acts noiselessly with respect to the Lipschitz semi-norm.
\item \emph{Amplification/scrambling}, \(\Lip(\Phi)>1\): $\Phi$ enlarges the spread of some Lipschitz observables, indicating information‑scrambling.
\end{itemize}
Mathematically, our goal is to find an appropriate functional $\Lip$ defined on channels $$\Lip: \mathfrak{C}(\mc N)\to \mb R_{\ge 0},$$
such that it satisfies the following axioms:
\begin{enumerate}
\item (\textit{Noiseless}) $\Lip(id) = 1$.
\item (\textit{Submultipilicative under concatenation}) $\Lip(\Phi \circ \Psi) \le \Lip(\Phi) \Lip(\Psi)$. 
\item (\textit{Convexity}) For any probability distribution $\{p_i\}_{i\in I}$ and any channels $\{\Phi_i\}_{i \in I}$, we have 
\begin{equation}
\Lip(\sum_{i\in I} p_i\Phi_i) \le \sum_{i\in I} p_i \Lip(\Phi_i).
\end{equation} 
\item (\textit{Universal upper bound}) $\Lip(\Phi)$ is upper bounded by a universal constant depending on the underlying dimension.
\item (\textit{Tensor maximality}) For finitely many channels $\{\Phi_i\}_{1\le i \le m}$, we have 
\begin{equation}
\Lip(\bigotimes_{i=1}^m \Phi_i) = \max_{1\le i \le m} \Lip(\Phi_i).
\end{equation}
\end{enumerate}
First we show that for a matrix Lipschitz semi-norm, the (complete) Lipschitz constant defined in \eqref{def: Lipschitz constant} and \eqref{def: complete Lipschitz constant} in Definition \ref{def:cost and constant} satisfies the first three axioms.
\begin{prop}
For a quantum channel $\Phi$, $\Lip_L(\Phi)$ and $\Lip^{cb}_L(\Phi)$ satisfy:
    \begin{enumerate}
\item (\textit{Noiseless}) $\Lip(id) = 1$.
\item (\textit{Submultipilicative under concatenation}) $\Lip(\Phi \circ \Psi) \le \Lip(\Phi) \Lip(\Psi)$. 
\item (\textit{Convexity}) For any probability distribution $\{p_i\}_{i\in I}$ and any channels $\{\Phi_i\}_{i \in I}$, we have 
\begin{equation}
\Lip(\sum_{i\in I} p_i\Phi_i) \le \sum_{i\in I} p_i \Lip(\Phi_i).
\end{equation} 
\end{enumerate}
\end{prop}
\begin{proof}
    Property $(1)$ follows from the definition. For property $(2)$, for any $n\ge 1$ and $x \in \mb M_n(\mc N)$, we have 
    \begin{equation}
        |||\Phi^{(n)}\circ \Psi^{(n)}(x)|||_{L^{(n)}} \le \Lip_{L^{(n)}}(\Phi^{(n)}) |||\Psi^{(n)}(x)|||_{L^{(n)}} \le \Lip_{L^{(n)}}(\Phi^{(n)}) \Lip_{L^{(n)}}(\Psi^{(n)}) |||x|||_{L^{(n)}}.
    \end{equation}
    Taking the supremum over all $x$ and $n\ge 1$, we conclude the proof. For property $(3)$, it follows from the definition and triangle inequality for semi-norm. 
\end{proof}
For the \textit{universal upper bound}, we need the same assumption as in Section \ref{sec: basic cost}. Using the matrix Lipschitz semi-norm given in \eqref{def: commutator seminorm}, we have
\begin{prop}[Universal upper bound]\label{main: universal bound constant}
    Under the Assumption \ref{assumption:resource}, for any channel $\Phi$ with $\Phi \circ E_\fix = E_\fix$ where $E_\fix$ is defined in \eqref{def: fixed point algebra}, we have  
    \begin{equation}
        \Lip_{\mc S}(\Phi) \le 2 \sup_{s\in \mc S}\|s\|_{op} \Cost_{\mc S}(E_\fix),\quad \Lip^{cb}_{\mc S}(\Phi) \le 2 \sup_{s\in \mc S}\|s\|_{op} \Cost^{cb}_{\mc S}(E_\fix)
    \end{equation}
\end{prop}
\begin{proof}
    For any $n \ge 1$ and $x\in \mb M_n(\mc N)$, by definition of $\Cost_{\mc S}$ defined in Definition \ref{def:cost and constant}, we have
\begin{equation}
   \|x - E_\fix^{(n)}(x)\|_{op} \le \Cost_{\mc S^{(n)}}(E_\fix)\,|||x|||_{\mc S^{(n)}}
\end{equation}
By definition of $\Lip_{\mc S^{(n)}}$, we have
\begin{align*}
    |||\Phi^{(n)}(x)|||_{\mc S^{(n)}}
    &= \sup_{s\in S} \Bigl\|\Bigl[\mathbf{1}_{\mathbb{M}_n}\otimes s,\,\Phi^{(n)}(x)(x)\Bigr]\Bigr\|_{\rm op} \\
    &= \sup_{s\in S} \Bigl\|\Bigl[\mathbf{1}_{\mathbb{M}_n}\otimes s,\,\Phi^{(n)}(x) - E_\fix^{(n)}(x)\Bigr]\Bigr\|_{\rm op} \\
    &\le 2\, \sup_{s \in S}\|s\|_{\rm op}\, \|\Phi^{(n)}(x) - E_\fix^{(n)}(x)\|_{\rm op} \\
    &\le 2\, \sup_{s \in S}\|s\|_{\rm op}\, \|x - E_\fix^{(n)}(x)\|_{op} \\
    & \le 2\, \sup_{s \in S}\|s\|_{\rm op}\, \Cost_{\mc S^{(n)}}(E_\fix)\,|||x|||_{\mc S^{(n)}}.
\end{align*}
In the second equality, we use the assumption that $E_\fix^{(n)}(x)$ commutes with $\mathbf{1}_{\mathbb{M}_n} \otimes \mc S$, since $E_\fix: \mc N\to \mc N\cap \mc S'$. The first inequality uses the definition of commutator and triangle inequality, and $\|s\|_{op} = \|\mathbf{1}_{\mathbb{M}_n}\otimes s\|_{op}$ The second inequality uses the contractivity of $\Phi^{(n)}$ under the operator norm and $\Phi^{(n)} \circ E_{\fix}^{(n)} = E_{\fix}^{(n)}$. Taking the supremum over $x$ and $n \ge 1$, we conclude the proof. 
\end{proof}
Using the same definition of matrix Lipschitz semi-norm on tensor product of von Neumann algebras in Section \ref{sec: basic cost}, we can prove the \textit{tensor maximality}:
\begin{prop}[Tensor maximality]\label{main:maximality}
Suppose $\Phi_j: \mc N_j \to \mc N_j$ are quantum channels and $\mc S_j \subset \mc B(\mc H_j)$ are the resource sets for each single system for any $1\le j \le d$, then 
 \begin{equation}\label{eqn: maximal tensorization}
        \Lip^{cb}_{\vee_{j=1}^d \mc S_j}\left( \bigotimes_{j=1}^d \Phi_j \right) \le \max_{1 \le i \le d}\, \Lip_{\wt {\mc S_j}}^{cb}(\wt {\Phi_j}),
    \end{equation}
where $\bigvee_{j=1}^d \mc S_j,\ \wt {\mc S_j}$ are defined in \eqref{def: composite resource} and $\wt{\Phi_j}$ is defined in \eqref{def: embedded channel}. Moreover, if $\Lip_{\wt {\mc S_j}}^{cb}(\wt {\Phi_j}) = \Lip_{\mc S_j}^{cb}(\Phi_j)$, we have equality in \eqref{eqn: maximal tensorization}.
\end{prop}
To prove the above result, we need a compatibility lemma:
\begin{lemma}\label{lemma:compatibility}
Suppose $\Phi_j: \mc N_j \to \mc N_j$ are quantum channels and $\mc S_j \subset \mc B(\mc H_j)$ are the resource sets for each single system for $j = 1,2$. Then for any $n\ge 1$ and $s_1\in \mc S_1, s_2\in \mc S_2$ and $x \in \mb M_n \otimes \mc N_1\otimes \mc N_2$,
    \begin{equation}
    \begin{aligned}
 & \wt{\Phi_1}^{(n)}\left((\mb I_n \otimes \textbf{1}_{\mc N_1} \otimes s_2)x \right) = (\mb I_n \otimes \textbf{1}_{\mc N_1} \otimes s_2) \cdot \wt{\Phi_1}^{(n)}\left(x \right), \\
  &  \wt{\Phi_2}^{(n)}\left((\mb I_n \otimes s_1 \otimes \textbf{1}_{\mc N_2})x \right) = (\mb I_n \otimes s_1 \otimes \textbf{1}_{\mc N_2}) \cdot \wt{\Phi_2}^{(n)}\left(x \right).
    \end{aligned}
    \end{equation}
\end{lemma}
\begin{proof}
    Via the well-known infinite-dimensional Stinespring’s dilation theorem for completely positive maps, see \cite[Theorem 4.1]{Paulsen2003}, there exist Hilbert spaces $\mc K_i$, isometries $V_i : \mc H_i \to \mc K_i$ (due to unitality of $\Phi_i$), unital $*$-homomorphisms $\pi_i: \mc N_i \to \mc B(\mc K_i)$, such that 
    \begin{equation}\label{eqn:Stinespring}
        \Phi_i(X_i) = V_i^*\pi(X_i) V_i,\quad X_i \in \mc N_i.
    \end{equation}
    Then we have 
    \begin{align*}
        & \wt{\Phi_1}^{(n)}\left((\mb I_n \otimes \textbf{1}_{\mc N_1} \otimes s_2)x \right) \\
        & = (id_{\mb M_n} \otimes \Phi_1 \otimes id_{\mc N_2}) \left( (\mb I_n \otimes \textbf{1}_{\mc N_1} \otimes s_2)x  \right)\\
        & = (\mb I_n \otimes V_1^* \otimes \textbf{1}_{\mc N_2})\cdot (id_{\mb M_n} \otimes \pi_1 \otimes id_{\mc N_2}) \left( (\mb I_n \otimes \textbf{1}_{\mc N_1} \otimes s_2)x  \right) \cdot (\mb I_n \otimes V_1 \otimes \textbf{1}_{\mc N_2}) \\
        & = (\mb I_n \otimes V_1^* \otimes \textbf{1}_{\mc N_2})\cdot (\mb I_n \otimes \pi_1(\textbf{1}_{\mc N_1}) \otimes s_2) \cdot \wt{\pi_1}^{(n)}(x)\cdot (\mb I_n \otimes V_1 \otimes \textbf{1}_{\mc N_2}) \\
        & = (\mb I_n \otimes V_1^* \otimes s_2)\cdot \wt{\pi_1}^{(n)}(x)\cdot (\mb I_n \otimes V_1 \otimes \textbf{1}_{\mc N_2}) \\
        & = (\mb I_n \otimes \textbf{1}_{\mc N_1} \otimes s_2)\cdot (\mb I_n \otimes V_1^* \otimes \textbf{1}_{\mc N_2})\cdot \wt{\pi_1}^{(n)}(x)\cdot (\mb I_n \otimes V_1 \otimes \textbf{1}_{\mc N_2}) \\
        & = (\mb I_n \otimes \textbf{1}_{\mc N_1} \otimes s_2)\cdot \wt{\Phi_1}^{(n)}(x).
    \end{align*}
    For the first equality, we use the definition of $\wt{\Phi_1}^{(n)}$. For the second equality, we use \eqref{eqn:Stinespring}. For the third equality, we use the fact that $\pi_1$ is a $*$-homomorphism. For the fourth equality, we use the fact that $\pi_1$ is unital. The remaining equalities follow from the definition. The same argument holds for $\Phi_2$ and we conclude the proof.
\end{proof}
Now we are ready to show the tensor maximality: 
\begin{proof}[Proof of Proposition \ref{main:maximality}]
     By induction, we only need to prove $d =2$.
By the definition of the joint resource set \eqref{def: composite resource}, for any $n \ge 1$, and $x \in \mb M_n(\mc N)$, we have
\begin{align*}
    & ||| (\Phi_1 \otimes \Phi_2)^{(n)}(x)|||_{(\mc S_1 \vee \mc S_2)^{(n)}} \\
    & = \max \left\{||| (\Phi_1 \otimes \Phi_2)^{(n)}(x)|||_{\wt{\mc S_1}^{(n)}}, ||| (\Phi_1 \otimes \Phi_2)^{(n)}(x)|||_{\wt{\mc S_2}^{(n)}} \right\} \\
    & = \max \left\{||| \wt{\Phi_2}^{(n)} \circ \wt{\Phi_1}^{(n)}(x)|||_{\wt{\mc S_1}^{(n)}}, ||| \wt{\Phi_1}^{(n)} \circ \wt{\Phi_2}^{(n)}(x)|||_{\wt{\mc S_2}^{(n)}} \right\},
\end{align*}
where the last equality follows from the definition of $\wt{\Phi_j}^{(n)}$ defined in \eqref{def: embedded channel}. Using the definition of $|||\cdot|||_{\wt{\mc S_1}^{(n)}}$, we have
\begin{align*}
    & ||| \wt{\Phi_2}^{(n)} \circ \wt{\Phi_1}^{(n)}(x)|||_{\wt{\mc S_1}^{(n)}} \\
    & = \sup_{s_1\in \mc S_1} \left\|\left[\mb I_n  \otimes s_1 \otimes \textbf{1}_{\mc B(\mc H_2)}, \wt{\Phi_2}^{(n)} \circ \wt{\Phi_1}^{(n)}(x) \right]\right\|_{op} \\
    & = \sup_{s_1\in \mc S_1} \left\|(\mb I_n  \otimes s_1 \otimes \textbf{1}_{\mc B(\mc H_2)}) \cdot (\wt{\Phi_2}^{(n)} \circ \wt{\Phi_1}^{(n)}(x) )
 - (\wt{\Phi_2}^{(n)} \circ \wt{\Phi_1}^{(n)}(x) )\cdot (\mb I_n  \otimes s_1 \otimes \textbf{1}_{\mc B(\mc H_2)})\right\|_{op} \\
 & = \sup_{s_1\in \mc S_1} \left\|\wt{\Phi_2}^{(n)}\left( (\mb I_n  \otimes s_1 \otimes \textbf{1}_{\mc B(\mc H_2)}) \cdot \wt{\Phi_1}^{(n)}(x) 
 -  \wt{\Phi_1}^{(n)}(x) \cdot (\mb I_n  \otimes s_1 \otimes \textbf{1}_{\mc B(\mc H_2)}) \right)\right\|_{op} \\
 & = \sup_{s_1\in \mc S_1} \left\| \wt{\Phi_2}^{(n)}\left( \big[ (\mb I_n \otimes s_1 \otimes \textbf{1}_{\mc B(\mc H_2)}), \wt{\Phi_1}^{(n)}(x) \big]\right)  \right\|_{op} \\
 & \le \sup_{s_1\in \mc S_1} \left\| \big[ (\mb I_n  \otimes s_1 \otimes \textbf{1}_{\mc B(\mc H_2)}), \wt{\Phi_1}^{(n)}(x) \big]\right\|_{op} = |||\wt{\Phi_1}^{(n)}(x)|||_{\wt{\mc S_1}^{(n)}}.
\end{align*}
For the third equality, we use Lemma \ref{lemma:compatibility}. 
Using the same argument, we can show that $$||| \wt{\Phi_1}^{(n)} \circ \wt{\Phi_2}^{(n)}(x)|||_{\wt{\mc S_2}^{(n)}} \le ||| \wt{\Phi_2}^{(n)}(x)|||_{\wt{\mc S_2}^{(n)}}.$$
Taking the supremum over $x$ and $n\ge 1$, we have 
\begin{align*}
     \Lip^{cb}_{\mc S_1 \vee \mc S_2}(\Phi_1 \otimes \Phi_2) \le \max\Bigl\{  \Lip_{\wt {\mc S_1}}^{cb}(\wt {\Phi_1}), \Lip_{\wt {\mc S_2}}^{cb}(\wt {\Phi_2})\Bigr\}.
\end{align*}
To show the other direction, for any $n_1,n_2 \in \mb N$ and $x_1 \in \mb M_{n_1}(\mc N_1), x_2 \in \mb M_{n_2}(\mc N_2)$, we construct 
\begin{equation}
    \wt{x_1} = x_1 \otimes \textbf{1}_{\mc N_2} \in \mb M_{n_1}(\mc N_1)  \otimes \mc N_2,\quad \wt{x_2} = \textbf{1}_{\mc N_1} \otimes x_2 \in \mc N_1 \otimes \mb M_{n_2}(\mc N_2).
\end{equation}
It is straightforward to show that 
\begin{align}\label{eqn:embed equality}
   |||x_j|||_{\mc S_j^{(n_j)}}= |||\wt{x_j}|||_{\wt{\mc S_j}^{(n_j)}} = |||\wt{x_j}|||_{(\mc S_1 \vee \mc S_2)^{(n_j)}}.
\end{align}
By definition, for any $j=1,2$, we have
\begin{align*}
    \Lip^{cb}_{\mc S_1 \vee \mc S_2}(\Phi_1 \otimes \Phi_2) &\ge \frac{|||(\Phi_1 \otimes \Phi_2)^{(n_j)}(\wt{x_j})|||_{(\mc S_1 \vee \mc S_2)^{(n_j)}}}{|||\wt{x_j}|||_{(\mc S_1 \vee \mc S_2)^{(n_j)}}} \\
    &= \frac{|||\wt{\Phi_j}^{(n_j)}(\wt{x_j})|||_{(\mc S_1 \vee \mc S_2)^{(n_j)}}}{|||\wt{x_j}|||_{(\mc S_1 \vee \mc S_2)^{(n_j)}}}  = \frac{|||\Phi_j^{(n_j)}(x_j)|||_{\mc S_j^{(n_j)}}}{|||x_j|||_{\mc S_j^{(n_j)}} }
\end{align*}
For the first equality, we use the unital property of $\Phi_j$ and for the last equality, we use the property \eqref{eqn:embed equality}. Taking the supremum over all $x_j$ and $n_j\ge 1$, we have 
\begin{align*}
    \Lip^{cb}_{\mc S_1 \vee \mc S_2}(\Phi_1 \otimes \Phi_2) \ge \max_{j=1,2} \Lip_{\mc S_j}^{cb}(\Phi_j)
\end{align*}
We conclude the proof by noting that $\Lip_{\wt {\mc S_j}}^{cb}(\wt {\Phi_j}) = \Lip_{\mc S_j}^{cb}(\Phi_j)$.
\end{proof}

\subsection{Continuity and mixing property}
In the remaining section, we discuss more useful properties of transportation cost and Lipschitz constant of channels. We take the special commutator semi-norm defined in \eqref{def: commutator seminorm}: for any $n\ge 1$, $|||x|||_{\mc S^{(n)}} : = \sup_{s \in \mc S} \left\|[\mb I_n\otimes s,x]\right\|_{op}$ with Assumption \ref{assumption:resource}.

We first show that both quantities are Lipschitz continuous with respect to $\|\cdot\|_{cb}$: 
\begin{prop}[Continuity property] \label{estimate:Lipschitz continuity}
Assume $\Phi_1,\Phi_2$ are two channels on $\mc N \subseteq \mc B(\mc H)$ and $\mc S \subseteq \mc B(\mc H)$ satisfies Assumption \ref{assumption:resource}. Moreover, assume
\begin{equation}
\Phi_i \circ E_{\fix} = E_{\fix},\ i =1,2.
\end{equation}
Then we have 
\begin{align}
    & |\Cost^{cb}_{\mc S}(\Phi_2) - \Cost^{cb}_{\mc S}(\Phi_1)| \le \kappa(\mc S) \|\Phi_1-\Phi_2\|_{cb}. \label{continuity:cost} \\
    & |\Lip_{\mc S}^{cb}(\Phi_2) - \Lip_{\mc S}^{cb}(\Phi_2)| \le 2 \sup_{s\in \mc S}\|s\|_{op}\,\kappa(\mc S) \|\Phi_1-\Phi_2\|_{cb}, \label{continuity:Lip constant}
\end{align}
where $\kappa(\mc S) = \CScb(E_\fix)$. 
\end{prop}
\begin{proof}
Note that $\Phi_i\circ E_{\fix} = E_{\fix}$, $i=1,2$. We have $(\Phi_2 - \Phi_1)\circ E_{\fix}=0$ and
\begin{align*}
\Phi_2 - id = \Phi_2 - \Phi_1 +\Phi_1 - id = (\Phi_2 - \Phi_1)\circ (id-E_{\fix}) +\Phi_1 - id.
\end{align*}
Therefore, for any $n \ge 1$ and $x=x^* \in \mb M_n(\mc N)$, we have 
\begin{align*}
    \|\Phi^{(n)}_2(x) - x\|_{op} & \le \|\Phi_1^{(n)}(x) - x\|_{op} + \|(\Phi^{(n)}_2 - \Phi^{(n)}_1)\circ (id-E^{(n)}_{\fix})(x)\|_{op} \\
    & \le \|\Phi_1^{(n)}(x) - x\|_{op} + \|E^{(n)}_{\fix}(x) - x\|_{op} \|\Phi^{(n)}_2 - \Phi^{(n)}_1\|_{op}.
\end{align*}
Taking the supremum over $x$ and $n\ge 1$, we have
\begin{align*}
    \Cost^{cb}_{\mc S}(\Phi_2) \le \Cost^{cb}_{\mc S}(\Phi_1) + \kappa(\mc S) \|\Phi_1-\Phi_2\|_{cb}.
\end{align*}
Switching the role of $\Phi_1$ and $\Phi_2$, we get \eqref{continuity:cost}. To show \eqref{continuity:Lip constant}, we use 
\begin{align*}
    \Phi_2 = \Phi_2 - \Phi_1 +\Phi_1 = (\Phi_2 - \Phi_1)\circ (id-E_{\fix}) +\Phi_1. 
\end{align*}
For any $n \ge 1$ and $x \in \mb M_n(\mc N)$, we have 
\begin{align*}
    |||\Phi_2^{(n)}(x)|||_{\mc S^{(n)}} & = |||(\Phi_2^{(n)} - \Phi_1^{(n)})\circ (id - E_\fix^{(n)})(x) + \Phi_1^{(n)}(x) |||_{\mc S^{(n)}} \\
    & \le 2 \sup_{s\in \mc S}\|s\|_{op} \|(\Phi_2^{(n)} - \Phi_1^{(n)})\circ (id - E_\fix^{(n)})(x) \|_{op} + |||\Phi_1^{(n)}(x)|||_{\mc S^{(n)}} \\
    & \le 2 \sup_{s\in \mc S}\|s\|_{op} \|E^{(n)}_{\fix}(x) - x\|_{op} \|\Phi^{(n)}_2 - \Phi^{(n)}_1\|_{op} + |||\Phi_1^{(n)}(x)|||_{\mc S^{(n)}}.
\end{align*}
Taking the supremum over $x$ and $n\ge 1$, we have
\begin{align*}
    \Lip^{cb}_{\mc S}(\Phi_2) \le \Lip^{cb}_{\mc S}(\Phi_1) + 2 \sup_{s\in \mc S}\|s\|_{op}\, \kappa(\mc S) \|\Phi_1-\Phi_2\|_{cb}.
\end{align*}
Switching the role of $\Phi_1$ and $\Phi_2$, we conclude the proof.
\end{proof}
Now we discuss the mixing property of channels. For any $\Phi$ with $\Phi\circ E_{\fix} = E_{\fix}$, it is natural to assume that $\Phi^n \to E_{\fix}$ in the sense that $\Phi^n(x) \to E_{\fix}(x)$ in weak-$^*$ topology, for any $x \in \mc N$. Here $\Phi^n$ is the $n$-th iteration: $\Phi^n:= \Phi \circ \Phi \circ \cdots \circ \Phi$. There are several notions of mixing time characterizing the convergence rate:
\begin{definition}
Suppose $\Phi \in \mathfrak{C}(\mc N)$, $\mc S \subseteq \mc B(\mc H)$ satisfies Assumption \ref{assumption:resource} and $\Phi \circ E_{\fix} = E_{\fix}$. Then we define 
\begin{enumerate}
        \item \textit{return time}: For any $\varepsilon \in (0,1)$,
        \begin{equation}
            t_{ret}(\varepsilon,\Phi):= \inf\{n\ge 1: (1-\varepsilon) E_{\fix} \le_{cp} \Phi^n \le_{cp} (1-\varepsilon) E_{\fix}\}.
        \end{equation}
        \item \textit{mixing time}: For any $\varepsilon \in (0,1)$,
        \begin{equation}\label{def:mixing time}
            t_{mix}(\varepsilon,\Phi):= \inf\{n\ge 1: \|\Phi^{n} - E_{\fix}\|_{cb} \le \varepsilon\}.
        \end{equation}
        \item \textit{cost induced mixing time}: For any $\varepsilon \in (0,1)$,
        \begin{equation}\label{def:cost mixing time}
            t_{mix}^{\mc S}(\varepsilon,\Phi):= \inf\{n\ge 1: \CScb(\Phi^n) \ge (1-\varepsilon) \kappa(\mc S)\}.
        \end{equation}
\end{enumerate}
\end{definition}
It is shown in \cite{GJLL22} that under some mild assumptions, we have $t_{mix}(\varepsilon,\Phi) \le t_{ret}(\varepsilon,\Phi)$. Moreover, using the continuity property \eqref{continuity:cost}, we have $t_{mix}^{\mc S}(\varepsilon,\Phi) \le t_{mix}(\varepsilon,\Phi)$. In summary, we have 
\begin{equation}
    t_{mix}^{\mc S}(\varepsilon,\Phi)  \le t_{mix}(\varepsilon,\Phi) \le t_{ret}(\varepsilon,\Phi)
\end{equation}
The following lower bound on the Lipschitz cost is useful: 
\begin{prop}\label{prop:lower bound cost}
Suppose $\Phi \in \mathfrak{C}(\mc N)$, $\mc S \subseteq \mc B(\mc H)$ satisfies Assumption \ref{assumption:resource} and $\Phi \circ E_{\fix} = E_{\fix}$. For any $\varepsilon\in (0,1)$, we have 
\begin{equation}
\CScb(\Phi) \ge \frac{1-\varepsilon}{t_{mix}^{\mc S}(\varepsilon,\Phi)} \kappa(\mc S). 
\end{equation}
\end{prop}
\begin{proof}
Without loss of generality, we assume $n=t^{\mc S}_{mix}(\varepsilon, \Phi)<\infty$. Using subadditivity under concatenation in Proposition \ref{main:elementary property}, we have 
\begin{align*}
    (1-\varepsilon)\CScb(E_{\fix}) \le \CScb(\Phi^n) \le n \CScb(\Phi) = t^{\mc S}_{mix}(\varepsilon, \Phi) \CScb(\Phi),
\end{align*}
which concludes the proof.
\end{proof}

\subsection{Universal upper bound on the expected length}
In this subsection, we provide a universal upper bound on the expected length $\kappa(\mc S)$ (Definition \ref{def: expected length}) for a given resource set $\mc S \subseteq \mc B(\mc H)$, which provides the largest transportation cost (Proposition~\ref{main:universal bound}) and contraction coefficient (Proposition~\ref{main: universal bound constant}) for any channels. 

We assume $\mc N= \mb M_d$, and postpone the infinite dimensional examples to Section \ref{sec:application cost}. Let $\mc S = \{V_j\}_{j=1}^m$ be a set of \footnote{In fact, it also works when $\mc S = \mc S^*$. For simplicity, we assume any operator is self-adjoint}{self-adjoint operators}. Our goal is to upper bound the expected length $\kappa(\mc S)$ defined in Definition \ref{def: expected length}. Our trick is to use the transportation-information inequality established in \cite{Fisher}. To be more specific, given $\mc S = \{V_j\}_{j=1}^m$, we construct a Lindbladian generator $\mc L:\mc N \to \mc N $ as a self-adjoint Lindbladian, i.e., 
\begin{equation}\label{Lindblad: continuous}
\mc L(x) = \sum_{j}\big(V_j\,x\, V_j - \frac{1}{2}(V_j^2x+xV_j^2)\big),\quad V_j \in \mc S,\ V_j = V_j^\dagger.
\end{equation}
In this setting 
\begin{equation}
    \mc N_{\fix} = \{X\in \mb M_d: [V_j,X]=0, \forall j\}, 
\end{equation}
and we choose $E_{\fix}$ as the trace-preserving conditional expectation onto $\mc N_{\fix}$. To establish an upper bound, we make use of \textit{complete modified logarithmic Sobolev inequality}(CLSI). Recall that $T_t:=\exp(t\mc L)$ satisfies the $\lambda$-modified logarithmic Sobolev inequality ($\lambda$-MLSI) if there exists a constant $\lambda > 0$ such that the relative entropy has exponential decay:
	\begin{equation}
		D(T_t(\rho) \| E_{\fix}\rho) \leq e^{-\lambda t}D(\rho \| E_{\fix}\rho)
	\end{equation}
for any state $\rho$, see \cite{wirth2025} for general von Neumann algebra setting. If the above inequality holds with $T_t,E_\fix$ replaced by $id_{\mb M_n} \otimes T_t,\ id_{\mb M_n} \otimes E_\fix$ and $\rho$ by $\rho_n \in S(\mb M_n(\mc N))$, we say it satisfies complete modified logarithmic Sobolev inequality. The largest constant $\lambda$ for which the relative entropy decays exponentially is denoted as $\text{MLSI}(\mc L)$ (and $\CLSI(\mc L)$ for the complete version).

A key inequality connecting the Lipschitz semi-norm and relative entropy (transportation-information inequality) is established in \cite[Corollary 6.8]{Fisher}:
\begin{prop}\label{proposition:TA}
		If the Lindbladian $\mc L$ generates a semigroup that satisfies $\lambda$-$\mathrm{CLSI}$, then we have the following transportation-information inequality:
		\begin{equation}
			||\rho_n - id_{\mb M_n} \otimes E_\fix(\rho_n)||_{\mc S^{(n)},2}^* \leq 4\sqrt{\frac{2D(\rho_n \| id_{\mb M_n} \otimes E_\fix(\rho_n))}{\lambda}}
		\end{equation} 
		for any state $\rho \in S(\mb M_n(\mc N)),\ n\ge 1$, where $|||\cdot|||_{\mc S,2}^*$ is dual to the Lipschitz semi-norm $|||\cdot|||_{\mc S,2}$ defined in \eqref{def: commutator seminorm}.
\end{prop}
An upper bound via transportation-information inequality is given as follows:
\begin{prop}\label{prop:expected length upper bound}
		The expected length $\kappa(\mc S) = \Cost_{\mc S}^{cb}(E_\fix)$ is upper bounded by 
		\begin{equation}
			 \kappa(\mc S) \le 4\sqrt{2m} \sqrt{\frac{\log \mc I^{cb}(E_\fix)}{\CLSI(\mc L)}},\quad m = |\mc S|
	\end{equation}
\end{prop}	
\begin{proof}
For notational simplicity, we only prove the case for $\Cost_{\mc S}(E_\fix)$ and the complete version follows in the same manner.
		
Using the relation between the two Lipschitz norms defined in \eqref{def: commutator seminorm}
\begin{equation}
    |||X|||_{\mc S,2} \le \sqrt{m} |||X|||_{\mc S}, 
\end{equation}
we have 
\begin{align*}
    \sup_{|||X|||_{\mc S} \le 1} \|E_\fix(X) - X\|_{op} & \le \sup_{|||X|||_{\mc S,2} \le \sqrt{m}} \|E_\fix(X) - X\|_{op} = \sqrt{m} \sup_{|||X|||_{\mc S,2} \le 1} \|E_\fix(X) - X\|_{op} \\
    & = \sqrt{m} \sup_{|||X|||_{\mc S,2} \le 1} \sup_{\rho \in S(\mc N)}\tr(\rho(E_\fix(X) - X)) \\
    & = \sqrt{m} \sup_{\rho \in S(\mc N)}\sup_{|||X|||_{\mc S,2} \le 1} 
 \tr((E_\fix(\rho) - \rho ) X) \\
  & = \sqrt{m} \sup_{\rho \in S(\mc N)} |||E_\fix(\rho) - \rho|||_{\mc S,2}^* \\
  & \le 4\sqrt{2m}\sqrt{\frac{\sup_\rho D(\rho \| E_\fix \rho)}{\mathrm{MLSI}(\mc L)}},
\end{align*}
where for the last equality, we use the definition of dual semi-norm and for the last inequality, we used Proposition \ref{proposition:TA}. 

Finally, we use the well-known fact that relative entropy is upper bounded by max-relative entropy and max-relative entropy is upper bounded by logarithmic of the index, see for example~\cite{GJL20entropy}, we have $\sup_\rho D(\rho \| E_\fix \rho) \le \log \mc I(E_\fix)$, which concludes the proof. 
\end{proof}
\begin{remark}
    It is shown in \cite{ding2024lower} that $\kappa(\mc S)$ is typically a lower bound for the simulation cost of Lindbladian system. Our upper bound indicates that when the number of jump operators is large, it is harder to simulate the system.
\end{remark}
\section{Applications of transportation cost} \label{sec:application cost}
\subsection{Expected group word length}\label{sec:finite group}
As a first application, we illustrate how our notion of cost, derived from the chosen set of generators, relates directly to the classical concept of word length in a finite group. Let $G$ be a finite group equipped with a finite generating set $S$. The \emph{word length} of an element $g \in G$ with respect to $S$ is defined as the minimum number of generators (or their inverses) required to express $g$. More precisely:
\[
    \ell_{S}(g) = \min \{ n \in \mathbb{N} : g = s_1 s_2 \cdots s_n, \; s_i \in S \cup S^{-1} \}.
\]
Without loss of generality, we assume the generating set $S \subseteq G$ is symmetric (i.e.\ $S=S^{-1}$). This function $\ell_S: G \to \mathbb{N}$ has several well-known properties:
\begin{itemize}
    \item $\ell_{S}(e) = 0$ for the identity element $e \in G$ (by convention).
    \item $\ell_{S}(g^{-1}) = \ell_{S}(g)$ for all $g \in G$.
    \item $\ell_{S}(gh) \leq \ell_{S}(g) + \ell_{S}(h)$ for all $g,h \in G$ (triangle inequality).
\end{itemize}
Note that the function $\ell_{S}$ depends both on the choice of generating set $S$ and the group structure of $G$.

We now reformulate these concepts in the von Neumann algebra setting. Consider the normalized Haar measure $\mu$ on the finite group $G$, i.e.\ the uniform distribution, so that $\mu(\{g\}) = |G|^{-1}$ for every $g \in G$. We define
\begin{equation}\label{def:von Neumann algebra finite group}
    \mathcal{N} = L_{\infty}(G,\mu) := \{ f : G \to \mathbb{C} \mid f \text{ bounded} \}, \quad \mathcal{H} = L_2(G,\mu),
\end{equation}
with the inner product 
\[
    \langle f_1, f_2\rangle_{\mathcal{H}} = \sum_{g \in G} f_1(g) \overline{f_2(g)} \mu(g).
\]
In this setting, $\mathcal{N}$ is a commutative von Neumann algebra. We obtain a representation 
\[
    \pi: \mathcal{N} \to \mathcal{B}(\mathcal{H}), \quad f \mapsto \pi(f),
\]
where $(\pi(f)\widetilde{f})(g) = f(g)\widetilde{f}(g)$ for every $\widetilde{f} \in \mathcal{H}$ and $g \in G$. Thus $\pi(f)$ acts as a pointwise multiplication operator on $L_2(G,\mu)$.

To measure cost in this operator framework, we use the left-regular representation:
\[
    \lambda: G \to \mathcal{B}(\mathcal{H}), \quad (\lambda_g \widetilde{f})(h) = \widetilde{f}(g^{-1}h), \quad \widetilde{f} \in \mathcal{H},\ h \in G.
\]
Choosing the set of resources as 
\[
    \mathcal{S} := \{\lambda_s : s \in S\} \subseteq \mathcal{B}(\mathcal{H}),
\]
we introduce a Lipschitz semi-norm on $\mathcal{N}$ by
\begin{equation}\label{def:Lipschitz norm finite group}
    ||| f |||_{\mathcal{S}} := \sup_{s \in S} \|[\lambda_s, \pi(f)]\|_{op}.
\end{equation}

This setting allows us to translate combinatorial aspects of the group (e.g.\ the word length and related metrics) into operator-theoretic language. In the next lemma, we will see how to calculate these operator norms and the corresponding Lipschitz semi-norm in the case of finite groups.
\begin{lemma}\label{lemma:calculation finite group}
    Let $G$ be a finite group with the normalized counting measure $\mu$, and let $\mc H = L_2(G,\mu)$. Suppose $f \in \mc N = L_{\infty}(G,\mu)$, where $(\pi,\mc H)$ is the standard representation defined by $\pi(f)(\wt f) = f\wt f$ for $\wt f \in \mc H$. Then:
    \begin{enumerate}
        \item We have
        \[
            \|\pi(f)\|_{op} = \|f\|_{\infty} = \sup_{g \in G} |f(g)|.
        \]
        \item For any symmetric generating set $S \subseteq G$ (i.e.\ $S=S^{-1}$), we have
        \[
            |||f|||_{\mc S} = \sup_{s\in S} \sup_{g \in G}|f(sg) - f(g)|.
        \]
    \end{enumerate}
\end{lemma}
\begin{proof}
(1) Since $\pi(f)$ acts by pointwise multiplication, for any $\wt f \in \mc H$ with $\|\wt f\|_{\mc H}=1$ we have
\[
    \|\pi(f)(\wt f)\|_{\mc H}^2 = \int_G |f(g) \wt f(g)|^2\,d\mu(g) \le \|f\|_{\infty}^2 \int_G |\wt f(g)|^2 \, d\mu(g) = \|f\|_{\infty}^2.
\]
This shows $\|\pi(f)\|_{op} \le \|f\|_{\infty}$.

To prove the reverse inequality, let $g_0 \in G$ be such that $|f(g_0)| = \|f\|_{\infty}$. Consider the function $\wt f := \sqrt{|G|} \cdot 1_{\{g_0\}}$, where $1_{\{g_0\}}$ is the indicator function of the singleton set $\{g_0\}$. Since $\mu(\{g_0\}) = 1/|G|$, we have
\[
    \|\wt f\|_{\mc H}^2 = \int_G |\wt f(g)|^2 d\mu(g) = |\wt f(g_0)|^2 \mu(\{g_0\}) = |G| \cdot 1 \cdot \frac{1}{|G|} = 1.
\]
Now,
\[
    \|\pi(f)(\wt f)\|_{\mc H}^2 = \int_G |f(g)\wt f(g)|^2 d\mu(g) = |f(g_0)|^2 |\wt f(g_0)|^2 \mu(\{g_0\}) = \|f\|_{\infty}^2.
\]
Hence $\|\pi(f)\|_{op} \ge \|f\|_{\infty}$. Combining both inequalities gives $\|\pi(f)\|_{op} = \|f\|_{\infty}$.

(2) Consider the commutator $[\lambda_s, \pi(f)]$ for $s \in S$. Here $\lambda_s$ is the left-regular representation acting by $(\lambda_s \wt f)(g) = \wt f(s^{-1}g)$. We compute:
\begin{align*}
    [\lambda_s, \pi(f)](\wt f)(g) 
    &= \lambda_s(\pi(f)\wt f)(g) - \pi(f)(\lambda_s \wt f)(g) \\
    &= (f\wt f)(s^{-1}g) - (f \cdot \lambda_s \wt f)(g) \\
    &= f(s^{-1}g)\wt f(s^{-1}g) - f(g)\wt f(s^{-1}g).
\end{align*}
By the left-invariance of the counting measure (and substituting $h = s^{-1}g$), we obtain
\[
    \|[\lambda_s,\pi(f)]\|_{op} = \sup_{\|\wt f\|_{\mc H}=1} \|[\lambda_s,\pi(f)]\wt f\|_{\mc H}.
\]
Note that for any $\|\wt f\|_{\mc H}=1$, we have
\begin{align*}
    \|[\lambda_s,\pi(f)]\wt f\|_{\mc H}^2 &= \int_G |f(s^{-1}g)\wt f(s^{-1}g)-f(g)\wt f(s^{-1}g)|^2\, d\mu(g) \\
    &= \int_G |f(h)-f(s h)|^2 |\wt f(h)|^2\, d\mu(h),
\end{align*}
where we used the substitution $h = s^{-1}g$. Since $\|\wt f\|_{\mc H}=1$, choosing $\wt f$ to be supported on the point $h_0$ where $|f(h_0)-f(s h_0)|$ attains its maximum shows
\[
    \|[\lambda_s,\pi(f)]\|_{op} \ge \sup_{h \in G} |f(h)-f(sh)|.
\]
On the other hand, for any $\wt f$,
\begin{align*}
    \|[\lambda_s,\pi(f)]\wt f\|_{\mc H}^2 
    &\le \sup_{h \in G}|f(h)-f(sh)|^2 \int_G |\wt f(h)|^2 d\mu(h)
    = \sup_{h \in G}|f(h)-f(sh)|^2,
\end{align*}
thus
\[
    \|[\lambda_s,\pi(f)]\|_{op} = \sup_{h \in G}|f(h)-f(sh)|.
\]
Since $|||f|||_{\mc S}$ is defined (in this context) as
\[
    |||f|||_{\mc S} = \sup_{s \in S} \|[\lambda_s,\pi(f)]\|_{op},
\]
we conclude that
\[
    |||f|||_{\mc S} = \sup_{s \in S} \sup_{g \in G}|f(sg)-f(g)|.
\]
This completes the proof.
\end{proof}
From the definition \eqref{def:Lipschitz norm finite group} and Lemma~\ref{lemma:calculation finite group}, it is immediate that $\mc S'$ consists precisely of the constant functions on $G$. Indeed, if $f \in \mc S'$, then for every $s \in S$, we have $[\lambda_s, \pi(f)] = 0$, which implies $f(sg)=f(g)$ for all $g \in G$. Since $S$ generates $G$, it follows that $f$ is constant. Consequently, the conditional expectation $E_{\fix}: \mc N \to \mc S'$ is given by
\begin{equation}\label{def:conditional_expectation_finite_group}
    E_{\fix}(f) = \left(\int_G f(g) \, d\mu(g)\right) 1_G.
\end{equation}

With this in hand, we can now characterize the expected group word length via the Lipschitz cost measure.

\begin{theorem}\label{thm:expected_length}
    Let $\mc N$, $\mc H$, and $\mc S$ be as in the finite group setup \eqref{def:von Neumann algebra finite group}, with the Lipschitz semi-norm \eqref{def:Lipschitz norm finite group}. Then we have
    \[
        \Cost_{\mc S}(E_{\fix})= \int_G \ell_{S}(g)\, d\mu(g).
    \]
\end{theorem}

\begin{proof}
\textbf{Lower bound:} Denote $\mc N_{\mathrm{s.a.}}$ as the set of real-valued functions. We need to show that 
\[
    \sup_{f \in \mc N_{\mathrm{s.a.}}, |||f|||_{\mc S} \le 1} \|\pi(E_{\fix}(f) - f)\|_{op} \ge \int_G \ell_{S}(g) \, d\mu(g).
\]

Consider the function 
\[
    f_{\mc S} := \ell_{S} - \left(\int_G \ell_{S}(g)\, d\mu(g)\right) 1_G.
\]
Note that $f_{\mc S}$ is real-valued and hence self-adjoint. By construction, $E_{\fix}(f_{\mc S}) = 0$. We now estimate its Lipschitz semi-norm:
\[
    |||f_{\mc S}|||_{\mc S} = \sup_{s \in S} \sup_{g \in G} |f_{\mc S}(sg) - f_{\mc S}(g)| = \sup_{s \in S} \sup_{g \in G} |\ell_{S}(sg) - \ell_{S}(g)|.
\]
Since adding or removing a single generator changes the word length by at most $1$, we have 
\[
    |\ell_{S}(sg) - \ell_{S}(g)| \le 1.
\]
Thus $|||f_{\mc S}|||_{\mc S} \le 1$.

Using Lemma~\ref{lemma:calculation finite group}, we then have
\[
    \|\pi(E_{\fix}(f_{\mc S}) - f_{\mc S})\|_{op} = \|\pi(-f_{\mc S})\|_{op} = \|f_{\mc S}\|_{\infty} = \sup_{g \in G} |f_{\mc S}(g)|.
\]
Since $f_{\mc S}(e) = \ell_{S}(e) - \int_G \ell_{S}(g)\, d\mu(g) = 0 - \int_G \ell_{S}(g)\, d\mu(g)$, we see that
\[
    \|f_{\mc S}\|_{\infty} \ge |f_{\mc S}(e)| = \int_G \ell_{S}(g)\, d\mu(g).
\]
This establishes the desired lower bound.

\medskip

\textbf{Upper bound:} Now we show
\[
    \sup_{f \in \mc N_{\mathrm{s.a.}}, |||f|||_{\mc S} \le 1} \|\pi(E_{\fix}(f) - f)\|_{op} \le \int_G \ell_{S}(g) \, d\mu(g).
\]

Take any $f \in \mc N_{\mathrm{s.a.}}$ with $|||f|||_{\mc S} \le 1$. Using the definition of $E_{\fix}$ and the invariance of $\mu$, we have:
\begin{align*}
    \|\pi(E_{\fix}(f) - f)\|_{op} &= \sup_{g \in G} \left| f(g) - \int_G f(h) \, d\mu(h) \right| \\
    &= \sup_{g \in G} \left| f(g) - \int_G f(hg) \, d\mu(h) \right| \\
    &\le \sup_{g \in G} \int_G |f(g) - f(hg)| \, d\mu(h).
\end{align*}

For a fixed $h \in G$, write $h = s_1 s_2 \cdots s_{\ell_S(h)}$ with $s_i \in S$. A telescoping sum argument yields
\[
    |f(g) - f(hg)| \le \sum_{k=1}^{\ell_S(h)} |f(u_k) - f(s_k u_k)|,
\]
where $u_k := s_{k+1}\cdots s_{\ell_S(h)}g$. Since $|||f|||_{\mc S} \le 1$, we have
\[
    |f(u_k) - f(s_k u_k)| \le \sup_{u \in G} |f(u) - f(s_k u)| \le |||f|||_{\mc S} \le 1.
\]
Thus
\[
    |f(g) - f(hg)| \le \ell_{S}(h).
\]
Substituting this back, we get
\[
    \|\pi(E_{\fix}(f) - f)\|_{op} \le \int_G \ell_{S}(h) \, d\mu(h).
\]
Since this holds for all $f$ with $|||f|||_{\mc S} \le 1$, we conclude
\[
    \sup_{f \in \mc N_{\mathrm{s.a.}}, |||f|||_{\mc S} \le 1} \|E_{\fix}(f) - f\|_{op} \le \int_G \ell_{S}(g) \, d\mu(g).
\]
Combining the lower and upper bounds completes the proof.
\end{proof}

\subsection{Carnot–Carath\'eodory distance}\label{sec:geometric}
The geometric approach to characterize the cost of a unitary operator by its distance to the identity operator is due to Nielsen \cite{N1,N2,N3}. The key idea in defining the geometric quantum cost measure is to regard the infinitesmial generators of the unitary operators as resources, or in other words, elements in the Lie algebra. It has a connection to sub-Riemannian geometry, as we review below. 
	
Suppose $G$ is a locally compact Lie group with Lie algebra $\mathfrak{g}$ and Haar measure $\mu$. Let $\mathfrak{h} \subseteq \mathfrak{g}$ be a linear subspace of the Lie algebra. Suppose $S = \{X_1,\dots,X_m\}$ is a linearly independent set spanning $\mathfrak{h}$, i.e. $\mathrm{span}\{X_1,\dots,X_m\} = \mathfrak{h}$. We endow $\mathfrak{h}$ with the inner product making $\{X_1,\dots,X_m\}$ an orthonormal set; for any $h \in \mathfrak{h}$ written as $h = \sum_{j=1}^m \alpha_j X_j$, we define
$$
\|h\|_{S,\mathfrak{h}} := \sqrt{\sum_{j=1}^m |\alpha_j|^2}.
$$

This choice of norm $\|\cdot\|_{S,\mathfrak{h}}$ on $\mathfrak{h}$ induces a so-called \textit{Carnot-Carathéodory (CC) distance} on the group $G$. To define this distance, we identify each tangent space $T_g G$ with $\mathfrak{g}$ via left translation: for $g \in G$, let $L_g: G \to G$ be the left-translation map $L_g(h) = gh$. The differential $d(L_g)_h$ of the left-translation map $L_g$ at a point $h \in G$ is defined as the tangent map from $T_h G$ to $T_{gh} G$. More precisely, if $\gamma: (-\varepsilon,\varepsilon) \to G$ is a smooth curve with $\gamma(0)=h$ and $\gamma'(0)=X \in T_h G$, then
$$
d(L_g)_h(X) := \frac{d}{dt}\bigg|_{t=0} L_g(\gamma(t)) = \frac{d}{dt}\bigg|_{t=0} (g\gamma(t)) \in T_{gh} G.
$$

In particular, when $h = e$ (the identity of $G$), we have a natural identification $T_e G \cong \mathfrak{g}$ and we denote $d(L_g)_e$ as $d(L_g)$. Consider a one-parameter subgroup $\gamma(t) = \exp(tX)$ with $\gamma(0)=e$ and $\gamma'(0)=X \in \mathfrak{g}$. Then
$$
d(L_g)(X) = \frac{d}{dt}\bigg|_{t=0} (g\exp(tX)) \in T_gG. $$

A piecewise smooth curve $\gamma: [0,1] \to G$ is said to be \textit{horizontal w.r.t. }$\mathfrak{h}\subseteq \mathfrak{g}$, if there exists a measurable function $X: [0,1] \to \mathfrak{h}$ such that for almost all $t \in [0,1]$,
$$
\gamma'(t) = d(L_{\gamma(t)})(X(t)).
$$
In other words, at each point $\gamma(t)$, the velocity $\gamma'(t)$ comes from the subspace $\mathfrak{h}$ of $\mathfrak{g}$.

Given two points $g,h \in G$, we define the following distance as
\begin{equation}\label{def:CC distance}
    d_\mathfrak{h}(g,h) := \inf \left\{ \int_0^1 \|X(t)\|_{S,\mathfrak{h}}\, dt : \gamma'(t) = d(L_{\gamma(t)})(X(t)),\ \gamma(0)=g,\ \gamma(1)=h \right\}.
\end{equation}
This definition endows \(G\) with a length metric that measures the shortest ``cost'' of traveling from \(g\) to \(h\) by curves whose directions are restricted to the subspace \(\mathfrak{h} \subseteq \mathfrak{g}\). Throughout this paper, we will assume $\mathfrak{h}$ is a H\"ormander system. In this case, $d_\mathfrak{h}$ defines a non-degenerate metric inducing a sub-Riemannian geometry on $G$. The resulting metric is known as the \textit{Carnot-Carath\'eodory distance}, satisfying several well-known properties \cite{Gromov1996}:
\begin{itemize}
    \item $d_\mathfrak{h}(g,g) = 0$ for any $g \in G$. 
    \item $d_\mathfrak{h}(g,h) = d_\mathfrak{h}(h,g)$ for all $g,h \in G$.
    \item $d_\mathfrak{h}(g_1,g_2) \le d_\mathfrak{h}(g_1,g_3) + d_\mathfrak{h}(g_2,g_3),\ \forall g_1,g_2,g_3\in G$ (triangle inequality).
    \item $d_\mathfrak{h}(hg_1,hg_2) = d_\mathfrak{h}(g_1,g_2),\ \forall g_1,g_2,h\in G$ (left-invariance). 
\end{itemize}

\begin{example}(Geometric cost \cite{N1,N2,N3})
Suppose $G = SU(2^n)$ and $\mathfrak{h}$ is taken to be the full Lie algebra $\mathfrak{su}(2^n)$. The orthonormal set is given by 
\begin{equation}
\{\sigma, p_{\sigma'} \sigma': \sigma \in \Sigma, \sigma'\in \Sigma'\},
\end{equation}
where $p_{\sigma'}>0$ is a weight constant, $\Sigma$ is the set of Pauli operators with weight less than or equal to $2$, $\Sigma'$ is the set of Pauli operators with weight greater than $2$. For any $U\in SU(2^n)$, the geometric cost is defined to be $d_\mathfrak{h}(g,e)$. See \cite{araiza2023} for a discussion of the relation between Lipschitz cost and geometric cost in this specific setting. 
\end{example}
Now we show that by choosing the Lipschitz norm accordingly, we can represent Carnot-Carath\'eodory distance $d_\mathfrak{h}(g,e)$ as the Lipschitz cost of a channel, similar to Section \ref{sec:finite group}.

To illustrate the idea, we discuss the commutative case first. Suppose $\mc N = L_{\infty}(G,\mu)$. For any $g\in G$ define $R_g: \mc N \to \mc N$ by
\begin{equation}
R_g(f)(h) = f(hg),\ f\in \mc N, h\in G.
\end{equation}
The Lipschitz semi-norm on $\mc N$ is given by 
\begin{equation}\label{def: Lipschitz commutative}
\|f\|_{\Lip,\mathfrak{h}}:= \sup_{g\neq h} \frac{|f(g) - f(h)|}{d_\mathfrak{h}(g,h)},
\end{equation}
and we denote $\mc N_{\Lip}:= \{f \in \mc N: \|f\|_{\Lip,\mathfrak{h}} < \infty\}$. Then $d_\mathfrak{h}(g,e)$ can be represented as the Lipchitz cost of $R_g$:
\begin{prop}\label{equality: commutative}
The Lipschitz cost of $R_g$ with respect to the Lipschitz semi-norm \eqref{def: Lipschitz commutative} is equal to $d_\mathfrak{h}(g,e)$:
\begin{equation}
   \Cost_\mathfrak{h}(R_g):= \sup_{f \in \mc N_{\Lip}, \|f\|_{\Lip,\mathfrak{h}} \le 1} \|R_g(f) - f\|_{\infty} = d_\mathfrak{h}(g,e).
\end{equation}
\end{prop}
\begin{proof}
For the upper bound, for any $f\in \mc N_{\Lip}$, we have 
\begin{align*}
\|R_g(f) - f\|_{\infty} = \sup_h |f(hg) - f(h)| \le \|f\|_{\Lip,\mathfrak{h}} \sup_h d_\mathfrak{h}(hg,h) = \|f\|_{\Lip,\mathfrak{h}} d_\mathfrak{h}(g,e),
\end{align*}
where the last equality used the left-invariance property of $d_\mathfrak{h}$. For the lower bound, we choose a function $f_0(g):= d_\mathfrak{h}(g,e)$ which has Lipschitz semi-norm less than 1 and we have 
\begin{align*}
\|(R_g- id)(f_0)\|_{\infty} \ge |R_g(f_0)(e) -f_0(e)| = d_\mathfrak{h}(g,e) \ge \|f_0\|_{\Lip,\mathfrak{h}} d_\mathfrak{h}(g,e).
\end{align*}
\end{proof}
To bring the problem to a noncommutative setting, suppose $\pi: G \to U(\mathcal{K})$ is a strongly continuous unitary representation on a Hilbert space $\mathcal{K}$. By differentiation, this induces a representation of the Lie algebra $\mathfrak{g}$ on the same Hilbert space:
$$
d\pi : \mathfrak{g} \to \mathfrak{u}(\mathcal{K}),
$$
where $\mathfrak{u}(\mathcal{K})$ denotes the skew-adjoint operators on $\mathcal{K}$. This Lie algebra representation is consistent with the group representation in the sense that for every $X \in \mathfrak{g}$,
$$
\pi(\exp(X)) = \exp(d\pi(X)).
$$
For any representation \begin{align*}
\pi: G \to U(\mc K),
\end{align*}
we can define a Lipschitz semi-norm on $\mc B(\mc K)$ corresponding to the representation and $\mathfrak{h} \subseteq \mathfrak{g}$.
\begin{definition}\label{Lip norm: representation}
Let $\mathfrak{h} \subseteq \mathfrak{g}$ be a subspace of the Lie-algebra and $X_1,\cdots, X_m$ be an orthonormal set such that $span\{X_1,\cdots, X_m\} = \mathfrak{h}$. The Lipschitz semi-norm corresponding to the representation $\pi: G \to U(\mc K)$ is given by 
\begin{equation}
\|x\|_{d\pi(\mathfrak{h})}:= \|\big(\sum_{j=1}^m [d\pi(X_j), x]^*[d\pi(X_j), x]\big)^{1/2}\|_{op},
\end{equation}
for any $x\in \text{Dom}_\mathfrak{h}:= \{y\in \mc B(\mc K): [d\pi(X_j), y]\in \mc B(\mc K),\ \forall 1\le j\le m\}$.
\end{definition}
For any channel $\Phi: \mc B(\mc K) \to \mc B(\mc K)$, we can define the Lipschitz cost 
\begin{equation}
\Cost{\mathfrak{h},\pi}(\Phi):= \sup_{x \in \text{Dom}_\mathfrak{h}, \|x\|_{d\pi(\mathfrak{h})} \le 1} \|\Phi(x) - x\|_{op}. 
\end{equation}
A simple observation is the Lipschitz cost $\Cost_{\mathfrak{h},\pi}(\adm_{\pi(g)})$ provides a lower bound for $d_\mathfrak{h}(g,e)$ for any representation $\pi$, where $\adm_{\pi(g)}: \mc B(\mc K) \to \mc B(\mc K)$ is defined as $\adm_{\pi(g)}(x) = \pi(g) x \pi(g)^*$. 
\begin{theorem}\label{main:CC distance lower bound}
For any $g \in G$, we have 
\begin{equation}
\Cost_{\mathfrak{h},\pi}(\adm_{\pi(g)}) \le d_\mathfrak{h}(g,e).
\end{equation}
\end{theorem}
\begin{proof}
Let $\gamma(t)$ be a piecewise differentiable horizontal path such that $$\gamma(0)= e, \gamma(1) = g,\quad \gamma'(t) = d(L_{\gamma(t)})(X(t)) $$ for some $X:[0,1]\to \mathfrak{h}$ and it achieves the infimum in $d_\mathfrak{h}(g,e)$. Consider for any $x\in \text{Dom}_\mathfrak{h}$,
\begin{equation}
x(t) := \pi(\gamma(t))^*x \pi(\gamma(t)) \in \mc B(\mc K).
\end{equation}
Using the formula connecting $\pi$ and $d\pi$, we have
\begin{align*}
    \frac{d}{dt}\pi(\gamma(t)) = \pi(\gamma(t))\, d\pi(X(t)).
\end{align*}
Hence the derivative of $x(t)$ is given by  
\begin{equation}\label{derivative of path}
\begin{aligned}
 x'(t)& = \big(\frac{d}{dt}\pi(\gamma(t))\big) x\pi(\gamma(t)) + \pi(\gamma(t))^* x \big(\frac{d}{dt}\pi(\gamma(t))^*\big)\\
 & = \pi(\gamma(t))d\pi(X(t))\, x\, \pi(\gamma(t))^* + \pi(\gamma(t))\, x\,d\pi(X(t))^*\, \pi(\gamma(t))^* \\
 & = \pi(\gamma(t))\, [d\pi(X(t)), x]\, \pi(\gamma(t))^*,
\end{aligned}
\end{equation}
where we used the fact that $d\pi(X(t))$ is a skew-adjoint operator. Then we have 
\begin{equation}\label{length: key step}
\begin{aligned}
\|x - \adm_{\pi(g)}(x)\|_{op} = \|x(1) - x(0)\|_{op} = \|\int_0^1 x'(t)dt\|_{op} \le \int_0^1 \|[d\pi(X(t)),x]\|_{op} dt.
\end{aligned}
\end{equation}
Since $X(t) \in \mathfrak{h} \subseteq \mathfrak{g}$, we can represent $X(t)$ as linear combination of $X_1,\cdots, X_m$, i.e., 
\begin{equation}
X(t) = \sum_{j=1}^m \alpha_j(t) X_j\,.
\end{equation}
Then by H\"older's inequality, we have 
\begin{equation}
\begin{aligned}
\|[d\pi(X(t)),x]\|_{op}& = \|\sum_{j=1}^m \alpha_j(t)[d\pi(X_j),x]\|_{op} \\
& \le (\sum_{j=1}^m|\alpha_j(t)|^2)^{1/2}\|(\sum_{j=1}^m |[d\pi(X_j),x]|^2)^{1/2}\|_{op} \\
&= \|X(t)\|_{S,\mathfrak{h}} \|x\|_{d\pi(\mathfrak{h})}.
\end{aligned}
\end{equation}
By \eqref{length: key step}, we have 
\begin{equation}
\|x - \adm_{\pi(g)}(x)\|_{op} \le \int_0^1 \|X(t)\|_{S,\mathfrak{h}} dt \|x\|_{d\pi(\mathfrak{h})}=  d_\mathfrak{h}(g,e)\|x\|_{d\pi(\mathfrak{h})},
\end{equation}
which concludes the proof.
\end{proof}	
\begin{remark}
    Readers who are familiar with sub-Riemannian geometry can show that the left regular representation saturates the inequality. The argument follows the same argument as the finite group, using the Lipschitz function given by the distance function to the unit. A technical difficulty is to extend the notion of $\ell_2$-type Lipschitz semi-norms to non-smooth functions. It can be done using the smoothing property of Laplace-Beltrami operator. For compact groups, we can simply restrict the discussion to functions given by matrix coefficients from finite dimensional representations.
\end{remark}


\section{Applications of Lipschitz constant}\label{sec:application Lipschitz constant}
In this section, we discuss some applications to quantum information theory. We assume the von Neumann algebra is the full matrix algebra $\mc N = \mb M_d$. In this case, $S(\mc N)$ is the set of density operators. Since we have a trace in this setting, we adopt the convention of a quantum channel being trace preserving and completely positive in this section. We adopt ketbra notation: $\ket{\psi}$ stands for a column vector and $\bra{\psi}$ for a row vector. 

For a quantum channel $\Phi$, we denote the Lipschitz constant of $\Phi$ for some Lipschitz semi-norm $|||\cdot|||_L$ as $\Lip_L(\Phi^*) = \sup_{X=X^\dagger} \frac{|||\Phi^*(X)|||_L}{|||X|||_L}$. 
\subsection{Entropy contraction from Lipschitz contraction}
The first application is to use Lipschitz constant to bound the entropy contraction coefficient for quantum channels. Suppose $\Phi$ is a trace preserving completely positive map and $\sigma \in S(\mc N)$ is a fixed density operator, define 
\begin{equation}\label{def:entropy contraction general}
    \eta^{\mb D}(\Phi,\sigma): =\sup_{\rho \in S(\mc N)} \frac{\mb D(\Phi(\rho) \| \sigma)}{\mb D(\rho \| \sigma)}.
\end{equation}
Here $\mb D(\cdot \| \cdot)$ is a divergence measure. We further assume that $\sigma$ is faithful and it is the unique fixed point state $\Phi(\sigma) = \sigma$. Our goal is to upper bound $\eta^{\mb D}(\Phi,\sigma)$ via the Lipschitz constant of $\Phi$ and get a sufficient condition for $\eta^{\mb D}(\Phi,\sigma) < 1$.

In this paper, we only consider the case when the divergence measure is the relative entropy and we leave the other divergence measures for future research. First recall the definition
\begin{equation}
      \RelEnt(\rho\Vert\sigma)
 =\tr \bigl[\rho(\log\rho-\log\sigma)\bigr],
  \quad \mathrm{supp}(\rho) \subseteq \mathrm{supp}(\sigma).
\end{equation}
We denote $\eta^D(\Phi,\sigma)$ as the contraction coefficient defined in \eqref{def:entropy contraction general} for Umegaki relative entropy. The elementary derivative formula is given by 
\begin{lemma}[Derivative of relative entropy]
\label{lem:derivative relative entropy}
Let $\sigma$ be a full–rank density operator
and $\rho(t)$ is a family of density operators which are smooth with respect to $t \in \mb R$. Then we have
\begin{equation}
    \left.\frac{d}{dt}\,D\bigl(\rho_t\Vert\sigma\bigr)\right|_{t=0}
 = \HS{\rho'(0)}{\log\rho(0) - \log \sigma}.
\end{equation}
\end{lemma}
\begin{proof}
    A well-known functional calculus for logarithmic function produces
    \begin{align*}
        \frac{d}{dt} \log \rho(t) = \int_0^\infty (\rho(t) + \lambda \textbf{1})^{-1} \rho'(t) (\rho(t) + \lambda \textbf{1})^{-1} dt.
    \end{align*}
    Therefore, 
    \begin{align*}
        \tr( \rho(0) \frac{d}{dt} \log \rho(t)\bigg|_{t = 0}) = \tr(\int_0^\infty \rho(0)(\rho(0) + \lambda \textbf{1})^{-1} \rho'(0) (\rho(0) + \lambda \textbf{1})^{-1} d\lambda) = \tr(\rho'(0)) = 0.
    \end{align*}
    As a result, by chain rule, we have 
    \begin{align*}
    \frac{d}{dt} D\bigl(\rho_t\Vert\sigma\bigr)\bigg|_{t=0} = \frac{d}{dt} \tr(\rho(t)\log \rho(t) - \rho(t)\log \sigma) \bigg|_{t=0} = \HS{\rho'(0)}{\log\rho(0) - \log \sigma}.
    \end{align*}
\end{proof}
For the supremum of $\eta^D(\Phi,\sigma)$, one has two cases:
\begin{enumerate}
    \item The suprema is achieved at a state $\rho \neq \sigma$.
    \item The suprema is achieved by a sequence of states $\rho_n$ and any convergent subsequence necessarily converges to $\sigma$. 
\end{enumerate}

\subsubsection{The suprema is achieved at a state $\rho \neq \sigma$} 
First we show that the optimizer must be a full-rank state if the channel $\Phi$ is strictly positive, i.e., $\Phi$ maps any state to a full-rank state.
\begin{lemma}\label{lemma:full-rank}
    Suppose the supremum of $\eta^D(\Phi,\sigma)$ is achieved at state $\rho \neq \sigma$:
    $$\eta:= \eta^D(\Phi,\sigma) = \frac{D(\Phi(\rho)\Vert\sigma)}
                {D(\rho\Vert\sigma)},\quad \rho \neq \sigma.$$
If $\Phi$ is strictly positive, then $\rho$ must be full-rank and
\begin{align}\label{eqn:optimizer equation}
    \Phi^*(\log \Phi(\rho) - \log \sigma) = \eta(\log \rho - \log \sigma).
\end{align}
\end{lemma}
\begin{proof}
First, we assume that $\rho$ is not full rank. Let $P_+$ denote the projection onto the support of $\rho$ and set $P_-:= \textbf{1} - P_+$. For any state $\wt \rho$ such that $\mathrm{supp}(\wt \rho) \subseteq P_-(\mc H)$, define 
\begin{equation}\label{eqn:ratio with time}
    \rho_t = (1-t)\rho + t \wt \rho, \quad f(t):=\frac{D(\Phi(\rho_{t})\Vert\sigma)}
                {D(\rho_{t}\Vert\sigma)},\quad t \ge 0.
\end{equation}
Since $\eta^D(\Phi,\sigma)$ is achieved at $\rho$, $\max_{t \in [0,1]}f(t) = f(0)$. Therefore, the right derivative at $0$ exists and it is a non-negative finite number
\begin{align*}
     f'(0+) = \frac{\HS{\Phi(\wt \rho - \rho)}{\log \Phi(\rho) - \log \sigma}-\eta\, \HS{\wt \rho - \rho}{\log \rho - \log \sigma}}{D(\rho \|\sigma)} \in (-\infty,0].
\end{align*}
Simplifying the above numerator using $D(\Phi(\rho)\|\sigma) = \eta D(\rho \|\sigma)$, we have 
\begin{align*}
    & \HS{\Phi(\wt \rho - \rho)}{\log \Phi(\rho) - \log \sigma}-\eta\, \HS{\wt \rho - \rho}{\log \rho - \log \sigma} \\
    & = \HS{\Phi(\wt \rho)}{\log \Phi(\rho) - \log \sigma}-\eta\, \HS{\wt \rho}{\log \rho - \log \sigma} \\
    & = \underbrace{\HS{\wt \rho}{\Phi^*(\log \Phi(\rho))}}_{\mathrm{possibly \ unbounded}} - \underbrace{\eta \HS{\wt \rho}{\log \rho}}_{\mathrm{possibly \ unbounded}} - \underbrace{\HS{\wt \rho}{\Phi^*(\log \sigma) - \eta \log \sigma}}_{\mathrm{bounded}}.
\end{align*}
For any pure state $\wt \rho = \ketbra{v}{v} \in P_-(\mc H)$, we have 
\begin{align*}
    \HS{\wt \rho}{\log \rho} & = \lim_{\varepsilon \to 0}\bra{v} \log (\rho + \varepsilon \textbf{1}) \ket{v} \\
    & =  \lim_{\varepsilon \to 0}\bra{v} \log (P_+ \rho P_+ + \varepsilon P_+ + \varepsilon P_-) \ket{v} \\
    & = \lim_{\varepsilon \to 0} \log \varepsilon \bra{v} P_- \ket{v} = -\infty.
\end{align*}
Since $\sigma$ has full rank, $\HS{\wt \rho}{\Phi^*(\log \sigma) - \eta \log \sigma}\in \mb R$. Moreover, using $\ f'(0+)\in \mb R$, we must have 
\begin{equation}\label{eqn:key-eqn minus infinity}
    \HS{\wt \rho}{\Phi^*(\log \Phi(\rho))} = \bra{v} \Phi^*\big( \log \Phi(\rho) \big) \ket{v} = -\infty.
\end{equation}
However, since $\Phi(\rho)$ is full-rank, $ \log \Phi(\rho)$ is a bounded operator implying that $\bra{v} \Phi^*\big( \log \Phi(\rho) \big) \ket{v} < \infty$ which is a contradiction. Thus $\rho$ must have full-rank.

To prove \eqref{eqn:optimizer equation}, for any traceless $X=X^{\dagger}$ consider $\rho_{t}=\rho+tX$ and define $f(t)$ as in \eqref{eqn:ratio with time}. Since $\rho$ is full-rank, there exists a small $\varepsilon>0$ such that $\rho(t)\ge 0$ for $t \in (-\varepsilon,\varepsilon)$. Under the assumption $\rho$ achieves the maximum of $\eta^D(\Phi,\sigma)$, we know that $f(t)$ achieves the maximum at $0$, implying $f'(0)=0$. Therefore, via Lemma \ref{lem:derivative relative entropy}, we have 
\begin{equation}
    f'(0) = \frac{\HS{\Phi(X)}{\log \Phi(\rho) - \log \sigma}-\eta\, \HS{X}{\log \rho - \log \sigma}}{D(\rho \|\sigma)},
\end{equation}
which implies that for any traceless $X=X^{\dagger}$, 
\begin{align*}
    \HS{X}{ \Phi^*(\log \Phi(\rho) - \log \sigma)} = \eta \HS{X}{\log \rho - \log \sigma}. 
\end{align*}
Since $X$ is arbitrary traceless Hermitian operator, we have 
\begin{equation}
     \Phi^*(\log \Phi(\rho) - \log \sigma) = \eta(\log \rho - \log \sigma) + \lambda \textbf{1}.
\end{equation}
Note that $D(\Phi(\rho)\|\sigma) = \eta D(\rho \|\sigma)$, by taking $\tr(\rho\, \cdot)$ on both sides, we have $\lambda = 0$, concluding the result. 
\end{proof}
In the case when $\Phi$ is not strictly positive, we only have a support relation for the optimizer.
\begin{prop}\label{prop:support relation}
    Suppose the supremum of $\eta^D(\Phi,\sigma)$ is achieved at state $\rho \neq \sigma$:
    $$\eta:= \eta^D(\Phi,\sigma) = \frac{D(\Phi(\rho)\Vert\sigma)}
                {D(\rho\Vert\sigma)},\quad \rho \neq \sigma.$$
Then we have \begin{align*}
    \mathrm{supp}(\rho)^\perp = \mathrm{Ker}(\rho) = \Phi^*(\mathrm{Ker}(\Phi(\rho))).
\end{align*}
\end{prop}
\begin{proof}
Let $P_+$ denote the projection onto the support of $\rho$ and set $P_-:= \textbf{1} - P_+$. Repeating the argument  in Lemma \ref{lemma:full-rank} on the reduced algebra \(P_{+}\,\mc N\, P_{+}\) gives
\begin{equation}\label{eqn:equality on the support}
    P_+ \Phi^*(\log \Phi(\rho) - \log \sigma) P_+ = \eta P_+ (\log \rho - \log \sigma)P_+.
\end{equation}
As in Lemma \ref{lemma:full-rank}, for any state $\wt \rho$ such that $\mathrm{supp}(\wt \rho) \subseteq P_-(\mc H)$, define $ \rho_t = (1-t)\rho + t \wt \rho, \quad t \ge 0$. Since $\eta^D(\Phi,\sigma)$ is achieved at $\rho$, $\max_{t \in [0,1]}f(t) = f(0)$, where $f(t):=\frac{D(\Phi(\rho_{t})\Vert\sigma)}{D(\rho_{t}\Vert\sigma)}$. Therefore, the right derivative at $0$ exists and it is a non-negative finite number
\begin{align*}
     f'(0+) = \frac{\HS{\Phi(\wt \rho - \rho)}{\log \Phi(\rho) - \log \sigma}-\eta\, \HS{\wt \rho - \rho}{\log \rho - \log \sigma}}{D(\rho \|\sigma)} \in (-\infty,0].
\end{align*}
Simplifying the above numerator using $D(\Phi(\rho)\|\sigma) = \eta D(\rho \|\sigma)$, we have 
\begin{align*}
    & \HS{\Phi(\wt \rho - \rho)}{\log \Phi(\rho) - \log \sigma}-\eta\, \HS{\wt \rho - \rho}{\log \rho - \log \sigma} \\
    & = \HS{\Phi(\wt \rho)}{\log \Phi(\rho) - \log \sigma}-\eta\, \HS{\wt \rho}{\log \rho - \log \sigma} \\
    & = \HS{\wt \rho}{\Phi^*(\log \Phi(\rho))} - \eta \HS{\wt \rho}{\log \rho} - \HS{\wt \rho}{\Phi^*(\log \sigma) - \eta \log \sigma}
\end{align*}
For any pure state $\wt \rho = \ketbra{v}{v} \in P_-(\mc H)$, we have 
$$\HS{\wt \rho}{\log \rho} = \lim_{\varepsilon \to 0}\bra{v} \log (\rho + \varepsilon \textbf{1}) \ket{v} = \lim_{\varepsilon \to 0} \log \varepsilon= -\infty.$$ 
Since $\sigma$ has full rank, $\HS{\wt \rho}{\Phi^*(\log \sigma) - \eta \log \sigma}\in \mb R$. Moreover, using $\ f'(0+)\in \mb R$, we must have 
\begin{equation}\label{eqn:key-eqn minus infinity2}
    \HS{\wt \rho}{\Phi^*(\log \Phi(\rho))} = \lim_{\varepsilon \to 0} \bra{v} \Phi^*\bigg( \log \big( \Phi(\rho ) + \varepsilon \textbf{1}\big) \bigg) \ket{v} = -\infty.
\end{equation}
Denote the spectral decomposition of $\Phi(\rho)$ as 
\begin{align*}
    \Phi(\rho) = \sum_{j=1}^d \lambda_j \ketbra{\psi_j}{\psi_j} = \sum_{j: \lambda_j>0} \lambda_j \ketbra{\psi_j}{\psi_j} + \sum_{j: \lambda_j = 0} 0\cdot \ketbra{\psi_j}{\psi_j}.
\end{align*}
Via \eqref{eqn:key-eqn minus infinity2}, we have 
\begin{align*}
    -\infty & = \lim_{\varepsilon \to 0} \bra{v} \Phi^*\bigg( \log \big( \Phi(\rho ) + \varepsilon \textbf{1}\big) \bigg) \ket{v} \\
    & = \lim_{\varepsilon \to 0} \bra{v} \Phi^*\bigg( \sum_{j=1}^d\log \big(  \lambda_j+\varepsilon \big) \ketbra{\psi_j}{\psi_j}  \bigg) \ket{v} \\
    & = \lim_{\varepsilon \to 0} \sum_{j=1}^d\log \big(  \lambda_j+\varepsilon \big) \bra{v} \Phi^*( \ketbra{\psi_j}{\psi_j}) \ket{v} \\
    & = \lim_{\varepsilon \to 0} \sum_{j: \lambda_j > 0}\log \big(  \lambda_j+\varepsilon \big) \bra{v} \Phi^*( \ketbra{\psi_j}{\psi_j}) \ket{v} + \lim_{\varepsilon \to 0} (\log \varepsilon)  \sum_{j: \lambda_j = 0}\bra{v} \Phi^*( \ketbra{\psi_j}{\psi_j}) \ket{v}
\end{align*}
Note that $\lim_{\varepsilon \to 0} \sum_{j: \lambda_j > 0}\log \big(  \lambda_j+\varepsilon \big) \bra{v} \Phi^*( \ketbra{\psi_j}{\psi_j}) \ket{v} $ is finite for any $\ket{v}$, we must have 
\begin{align*}
    \sum_{j: \lambda_j = 0}\bra{v} \Phi^*( \ketbra{\psi_j}{\psi_j}) \ket{v} \neq 0,\quad \forall \ket{v} \in P_-(\mc H).
\end{align*}
Since the pure state $\ket{v} \in P_-(\mc H)$ is arbitrary, we have 
\begin{align*}
    P_-(\mc H) \subseteq \mathrm{supp}\left(\Phi^*\big(\sum_{j: \lambda_j = 0} \ketbra{\psi_j}{\psi_j} \big) \right).
\end{align*}
Using a similar limit argument as above and \eqref{eqn:equality on the support}, we have 
\begin{align*}
    & P_+ \Phi^*(\log \Phi(\rho) - \log \sigma) P_+ = \eta P_+ (\log \rho - \log \sigma)P_+\\
    & = \sum_{j: \lambda_j >0} \log \lambda_j\cdot P_+ \Phi^*(\ketbra{\psi_j}{\psi_j})P_+ + \lim_{\varepsilon\to 0}(\log \varepsilon)\cdot P_+\Phi^*\big(\sum_{j: \lambda_j = 0} \ketbra{\psi_j}{\psi_j} \big)P_+ - P_+ \Phi^*(\log \sigma)P_+.
\end{align*}
Note that $\eta P_+ (\log \rho - \log \sigma)P_+$, $ \sum_{j: \lambda_j >0} \log \lambda_j\cdot P_+ \Phi^*(\ketbra{\psi_j}{\psi_j})P_+$, and $P_+ \Phi^*(\log \sigma)P_+$ are finite operators, we must have \begin{align*}
    P_+\Phi^*\big(\sum_{j: \lambda_j = 0} \ketbra{\psi_j}{\psi_j} \big)P_+ = 0,
\end{align*}
implying that 
\begin{align*}
    P_-(\mc H) = \mathrm{supp}\left(\Phi^*\big(\sum_{j: \lambda_j = 0} \ketbra{\psi_j}{\psi_j} \big) \right),
\end{align*}
which concludes the result.
\end{proof}
\begin{remark}
    An immediate corollary of Proposition~\ref{prop:support relation} shows that \eqref{eqn:optimizer equation} holds on the diagonals, i.e., 
    \begin{align*}
        & P_+ \Phi^*(\log \Phi(\rho) - \log \sigma) P_+= \eta\, P_+(\log \rho - \log \sigma)P_+,\\
        & P_- \Phi^*(\log \Phi(\rho) - \log \sigma) P_-= \eta\, P_-(\log \rho - \log \sigma)P_-,
    \end{align*}
    which recovers the classical case, see \cite[Lemma 2]{caputo2025entropy}. However, it is not clear whether the off-diagonal terms are equal.
\end{remark}

\subsubsection{The suprema is achieved by a sequence of states $\rho_n$ and any convergent subsequence necessarily converges to $\sigma$. }
In this case, we have 
\begin{align*}
    \eta^D(\Phi,\sigma) = \lim_{n\to \infty} \frac{D(\Phi(\rho_n) \| \sigma)}{D(\rho_n \| \sigma)}. 
\end{align*}
Without loss of generality, we assume that $\rho_n \to \sigma$. Then we have a bounded sequence of traceless Hermitian operators $X_n$ such that 
\begin{equation}
    \rho_n = \sigma + \varepsilon_n X_n, \quad \varepsilon_n\to 0.
\end{equation}
Using the well-known second order asymptotics of relative entropy, see for example (\cite{LGCR} or \cite[Section 4.1]{belzig2024}), we have 
\begin{align}\label{eqn:second order relative entropy}
    D(\rho_n \| \sigma)& = D(\sigma + \varepsilon_n X_n \| \sigma)= \varepsilon_n^2 \tr\left(\int_0^\infty X_n(\sigma + r\textbf{1})^{-1} X_n(\sigma + r\textbf{1})^{-1}dr\right) + o(\varepsilon_n^2).
\end{align}
Denote the Bogoliubov-Kubo-Mori(BKM) inner product as 
\begin{equation}
    \BKM{X}{Y}:= \HS{X}{\Gamma_\sigma(Y)},\quad \Gamma_\sigma(Y) :=\int_0^1 \sigma^s Y \sigma^{1-s} ds.
\end{equation}
Using the scalar integrals 
\begin{align*}
    \int_0^1 a^s b^{1-s}ds = \frac{a-b}{\log a - \log b},\quad \int_0^\infty \frac{1}{(a+r)(b+r)}dr = \frac{\log a - \log b}{a-b},
\end{align*}
and functional calculus for left and right multiplication operator, we have $\Gamma_\sigma^{-1}(X)= \int_0^\infty(\sigma + r\textbf{1})^{-1} X_n(\sigma + r\textbf{1})^{-1}dr$. 

Then define $h_n = \Gamma_\sigma^{-1}(X_n)$, one has 
\begin{align*}
    \eta^D(\Phi,\sigma) = \lim_{n\to \infty} \frac{D(\Phi(\rho_n) \| \sigma)}{D(\rho_n \| \sigma)} & = \lim_{n\to \infty} \frac{\HS{\Phi(X_n)}{\Gamma_\sigma^{-1}\circ \Phi(X_n)}}{\HS{X_n}{\Gamma_\sigma^{-1}(X_n)}} \\
    & = \frac{\BKM{h_n}{\Phi^*\circ \Gamma_\sigma^{-1} \circ \Phi \circ \Gamma_\sigma(h_n)}}{\BKM{h_n}{h_n}}
\end{align*}
Denote the BKM dual as 
\begin{equation}\label{def:BKM dual}
    \Phi^{*,\mathrm{BKM}}:= \Gamma_\sigma^{-1} \circ \Phi \circ \Gamma_\sigma.
\end{equation}
Then it is routine to check that under the assumption that $\Phi$ has a unique full-rank fixed state $\sigma$, $\Phi^* \circ \Phi^{*,\mathrm{BKM}}$ is:
\begin{itemize}
    \item $\mathrm{BKM}$-symmetric, i.e., $\BKM{\Phi^* \circ \Phi^{*,\mathrm{BKM}}(X)}{Y} = \BKM{X}{\Phi^* \circ \Phi^{*,\mathrm{BKM}}(Y)}$.
    \item positive, i.e., for any operator $X$, $\BKM{\Phi^* \circ \Phi^{*,\mathrm{BKM}}(X)}{X} \ge 0$. 
    \item ergodic, i.e., the fixed point space is given by $\{c \textbf{1}: c \in \mb C\}$.
\end{itemize}
In the quantum literature, ergodic is also called primitive for the channel. Therefore, $\Phi^* \circ \Phi^{*, \mathrm{BKM}}$ can be diagonalized in the space $L^2(\mc N, \BKM{\cdot}{\cdot})$, i.e., there exist a basis $\{f_j\}_{j=1}^{d^2}$ for $L^2(\mc N, \BKM{\cdot}{\cdot})$ and $\{\lambda_j\}_{j=1}^{d^2}$ with \begin{align*}
    1 = \lambda_1 > \lambda_2 \ge \cdots \ge \lambda_{d^2} \ge 0,
\end{align*}
such that 
\begin{align*}
    \Phi^* \circ \Phi^{*, \mathrm{BKM}}(f_j) = \lambda_j f_j.
\end{align*}
For $h_n = \Gamma_\sigma^{-1}(X_n)$, there exists $c_j$ such that $ h_n = \sum_j c_j f_j$ and 
\begin{equation}
   \BKM{\Phi^* \circ \Phi^{*, \mathrm{BKM}}(h_n) }{h_n} = \sum_{j=1}^{d^2} c_j \lambda_j \BKM{f_j}{h_n}  \quad \BKM{h_n}{h_n} = \sum_{j=1}^{d^2} c_j \BKM{f_j}{h_n} .
\end{equation}
Note that $f_1 = \textbf{1}$, we have $\BKM{f_1}{h_n} = \tr(X_n) = 0$, therefore
\begin{align*}
    \frac{\BKM{\Phi^* \circ \Phi^{*, \mathrm{BKM}}(h_n)}{h_n}}{\BKM{h_n}{h_n}} & = \frac{\sum_{j=1}^{d^2} c_j \lambda_j \BKM{f_j}{h_n}}{ \sum_{j=1}^{d^2} c_j \BKM{f_j}{h_n}} = \frac{\sum_{j=2}^{d^2} c_j \lambda_j \BKM{f_j}{h_n}}{\sum_{j=2}^{d^2} c_j \BKM{f_j}{h_n}} \\
    & \le \lambda_2 \frac{\sum_{j=2}^{d^2} c_j \BKM{f_j}{h_n}}{\sum_{j=2}^{d^2} c_j \BKM{f_j}{h_n}} = \lambda_2.
\end{align*}
Using the fact that $\Phi^* \circ \Phi^{*, \mathrm{BKM}}(f_2) = \lambda_2 f_2$, for a Lipschitz semi-norm $|||\cdot|||_L$, we have
\begin{equation}
    \Lip_L(\Phi^* \circ \Phi^{*, \mathrm{BKM}}) \ge \frac{|||\Phi^* \circ \Phi^{*, \mathrm{BKM}}(f_2)|||_L}{|||f_2|||_L} = \lambda_2.
\end{equation}
Note that the last part of the above argument obviously goes through in infinite dimension. In summary, we show that 
\begin{lemma}\label{lemma:converging sequence}
     If $\eta^D(\Phi,\sigma)$ is achieved by a sequence of states $\rho_n$ and any convergent subsequence necessarily converges to $\sigma$, then we have 
     \begin{equation}
         \eta^D(\Phi,\sigma) \le \Lip_L(\Phi^* \circ \Phi^{*, \mathrm{BKM}}). 
     \end{equation}
\end{lemma}
Combine the two cases, we have the following theorem 
\begin{theorem}\label{main: entropy contraction upper}
    Suppose $\Phi$ is strictly positive, the entropy contraction coefficient satisifes
    \begin{equation}
        \eta^D(\Phi,\sigma) \le \max\{\Lip_L(\Phi^* \circ \Phi^{*, \mathrm{BKM}}), \Lip_L(\Phi^*) \sup_{\rho \in S(\mc N), \rho>0}\frac{|||\log \Phi(\rho) - \log \sigma|||_L}{|||\log \rho - \log \sigma|||_L}\}
    \end{equation}
    for any Lipschitz semi-norm $|||\cdot|||_L$.
\end{theorem}
Before we proceed, we have the following remark:
\begin{remark}
When the channel is not stricly positive, one can choose $\Phi_\varepsilon:= (1-\varepsilon)\Phi + \varepsilon \mc R_\sigma$, where $\mc R_\sigma(X):= \tr(X) \sigma$ is the replacer channel. Then $\Phi_\varepsilon$ is strictly positive and the only extra technicality is 
    \begin{align*}
        \limsup_{\varepsilon \to 0}\frac{|||\log \left( (1-\varepsilon)\Phi(\rho_\varepsilon) + \varepsilon\sigma \right) - \log \sigma|||_L}{|||\log \rho_\varepsilon - \log \sigma|||_L}.
    \end{align*}
When the channel is unital, $\sigma$ is maximally mixed state, thus by unit degeneracy property of $|||\cdot|||_L$, we have  
    \begin{equation}
        \mathrm{Log}\Lip_L(\Phi,\textbf{1}/d) = \sup_{\rho \in S(\mc N), \rho>0}\frac{|||\log \Phi(\rho) |||_L}{|||\log \rho |||_L}.
    \end{equation}
The technical part is how to deal with $\log \Phi(\rho)$. 
\end{remark}

\subsubsection{Concrete estimates via BMO norms}
Using Theorem \ref{main: entropy contraction upper}, we need to get estimates for $\Lip_L(\Phi^*)$, which was extensively discussed in \cite{gaorouze24}. Thus, we focus on the term 
\begin{align}\label{def:log Lipschitz constant}
    \mathrm{Log}\Lip_L(\Phi,\sigma):= \sup_{\rho \in S(\mc N), \rho>0}\frac{|||\log \Phi(\rho) - \log \sigma|||_L}{|||\log \rho - \log \sigma|||_L}.
\end{align}
When the channel is unital, that is, $\sigma = \textbf{1}/d$, we have $\mathrm{Log}\Lip_L(\Phi,\textbf{1}/d) = \frac{|||\log \Phi(\rho)|||_L}{|||\log \rho|||_L}$. 

The main technical tool we use is the bounded mean oscillation (BMO) seminorm. Given a symmetric quantum Markov semigroup $T_t: \mc N \to \mc N$, for any $X \in \mc N$ with $\lim_{t\to \infty} T_t(X) = 0$, define 
    \begin{equation}
        \|X\|_{\mathrm{BMO}, T_t}:= \sup_{t\ge 0} \max\{ \|T_t(X^*X) - T_t(X^*)T_t(X)\|_{op}^{1/2}, \|T_t(XX^*) - T_t(X)T_t(X^*)\|_{op}^{1/2} \}.
    \end{equation}
Note that the symmetric quantum Markov semigroup on $\mc N$ inherits a Markov dilation from the standard Markov dilation of the heat semigroup on $L^\infty(\mb R^2;\mc N)$. Then one has the following well-established relation between BMO semi-norm and \textbf{normalized} Schatten $p$-norms, see \cite{junge2007noncommutative} for discrete time version and \cite[Theorem 0.2]{junge2012bmo} for the continuous time semigroup version: 
\begin{equation}\label{BMO to Schatten p}
    \|X\|_p \lesssim p \|X\|_{\mathrm{BMO}, T_t},\quad \forall X:\lim_{t\to \infty} T_t(X) = 0. 
\end{equation}
In the following, we denote $a \lesssim b$ if there exists a universal constant $c_{abs}>0$ such that $a \le c_{abs}\cdot b$. 

The Lipschitz estimate established in \cite{caspers2020bmo} is useful:
\begin{lemma}\label{lemma:BMO}
    For any self-adjoint $A \in \mc N$ and $f: \mb R \to \mb R$ is a Lipschitz function on the spectrum of $A$, there is a symmetric quantum Markov semigroup $T_t: \mc N \to \mc N$ and a universal constant $c>0$ such that 
    \begin{equation}
        \|[f(A),x]\|_{\mathrm{BMO}, T_t} \le c \|f\big|_{\mathrm{spec}(A)}\|_{\Lip} \|[A,x] \|_{\infty},\ \forall x \in \mc N.
    \end{equation}
\end{lemma}
The above estimate provides a universal upper bound on the log Lipschitz constant defined in \eqref{def:log Lipschitz constant}, when the channel is unital:
\begin{prop}\label{prop:upper bound on log Lipschitz}
    For any commutator Lipschitz semi-norm and unital, strictly positive channel $\Phi$, we have an absolute constant $c_{abs}>0$ such that 
    \begin{equation}
        \mathrm{Log}\Lip_L(\Phi,\textbf{1}/d) \le c_{abs} \Lip_L(\Phi)\frac{\log^2d }{\lambda_{\min}(\Phi)},
    \end{equation}
    where $\lambda_{\min}(\Phi)$ is the minimal eigenvalue of $\Phi(\rho)$ over any state $\rho$.
\end{prop}

\begin{proof}
Using the comparison between the operator and normalized Schatten norm $\|\cdot\|_{op} \le d^{1/p} \|\cdot\|_p$, the comparison between the normalized Schatten norm and BMO semi-norms \eqref{BMO to Schatten p} and Lemma \ref{lemma:BMO}, we have
\begin{align*}
    \|[\log \Phi(\rho), x]\|_{op} & \le d^{1/p} \|[\log \Phi(\rho), x]\|_{p} \\
    & \lesssim d^{1/p} p \|[\log \Phi(\rho), x]\|_{\mathrm{BMO},T_t} \\
    & \lesssim \frac{d^{1/p} p}{\lambda_{\min}(\Phi)} \|[\Phi(\rho),x]\|_{op} \\
    & \le \frac{d^{1/p} p}{\lambda_{\min}(\Phi)} \Lip_{\{x\}}(\Phi) \|[\rho,x]\|_{op} \\
    & \le \frac{d^{2/p} p}{\lambda_{\min}(\Phi)} \Lip_{\{x\}}(\Phi) \|[\rho,x]\|_{p} \\
    & \lesssim \frac{d^{2/p} p^2}{\lambda_{\min}(\Phi)} \Lip_{\{x\}}(\Phi) \|[\rho,x]\|_{\mathrm{BMO},T_t} \\
    & = \frac{d^{2/p} p^2}{\lambda_{\min}(\Phi)} \Lip_{\{x\}}(\Phi) \|[\exp(\log \rho),x]\|_{\mathrm{BMO},T_t} \\
    & \lesssim \frac{d^{2/p} p^2}{\lambda_{\min}(\Phi)} \Lip_{\{x\}}(\Phi) \|[\log \rho,x]\|_{op}.
\end{align*}
Finally, we choose $p = \log d$ to conclude the proof.
\end{proof}
Here we remark that the upper estimate is not sharp but works for general channels. We will discuss its implications on mixing time estimate in Section \ref{sec:mixing time estimate}.

\subsubsection{An example} 
In this subsection, we show that for some nice channels, $\mathrm{LogLip}(\Phi,\textbf{1}/d)$ can be explicitly calculated and it implies sharp entropy contraction coefficients. We consider the depolarizing channel
\begin{equation}
    \Phi_{p,d}(\rho)=(1-p)\rho+\tfrac{p}{d} \textbf{1},\quad p\in [0,1],\ d\ge2.
\end{equation}
Note that $\Phi_{p,d}$ is symmetric, i.e., $\Phi_{p,d} = \Phi_{p,d}^*$. For any Lipschitz seminorm $|||\cdot|||_L$, via unit degeneracy property, one has 
\begin{equation}\label{eqn:depolarizing property}
    |||\Phi_{p,d}(X)|||_L = (1-p)|||X|||_L,\quad \forall X.
\end{equation}
Therefore, we have $\Lip_L(\Phi_{p,d}) = 1-p$. We aim to get an explicit upper bound on the entropy contraction coefficient $\eta^D(\Phi_{p,d}, \textbf{1}/d)$, here $\textbf{1}/d$ is the unique fixed point of $\Phi_{p,d}$.  First recall that if $\eta^D(\Phi_{p,d}, \textbf{1}/d)$ is achieved by a sequence of states $\rho_n$ and any convergent subsequence necessarily converges to $\textbf{1}/d$, via Lemma \ref{lemma:converging sequence}, we have 
\begin{equation}\label{eqn:upper 1 depolarizing}
     \lim_{n\to \infty}\frac{D(\Phi_{p,d}(\rho_n)\| \textbf{1}/d)}{D(\rho_n\| \textbf{1}/d)} = \eta^D(\Phi_{p,d}, \textbf{1}/d) \le \Lip_L(\Phi_{p,d})^2 = (1-p)^2.
\end{equation}
If $\eta^D(\Phi_{p,d}, \textbf{1}/d)$ is achieved by a state $\rho \neq \textbf{1}/d$, via Lemma \ref{lemma:full-rank}, $\rho$ has full rank and satisfy
\begin{equation}
    \Phi_{p,d}(\log \Phi_{p,d}(\rho) - \log \textbf{1}/d) = \eta(\log \rho - \log \textbf{1}/d). 
\end{equation}
Then for any Lipschitz semi-norm $|||\cdot|||_L$, we have 
\begin{align*}
    |||\Phi_{p,d}\left(\log \Phi_{p,d}(\rho) - \log \textbf{1}/d\right )|||_L= \eta  |||\log \rho - \log \textbf{1}/d|||_L.
\end{align*}
Via \eqref{eqn:depolarizing property} and unit degeneracy property for $|||\cdot|||_L$, we have 
\begin{align}\label{eqn:upper 2 depolarizing}
    (1-p)|||\log \Phi_{p,d}(\rho)|||_L = |||\Phi_{p,d}\left(\log \Phi_{p,d}(\rho)\right) |||_L= \eta  |||\log \rho|||_L.
\end{align}
To get an explicit upper bound on $\eta$, we consider the qubit case, i.e., $d = 2$. A concrete estimate is given as follows:
\begin{lemma}
For any $p\in [0,1]$, one has 
\begin{equation}
    |||\log \Phi_{p,2}(\rho)|||_S \le (1-p) |||\log \rho|||_S
\end{equation}
for the full-rank density $\rho$ that achieves $\eta^D(\Phi_{p,d},\textbf{1}/2)$.
\end{lemma}
\begin{proof}
For the full-rank density operator $\rho$, since $d = 2$, the spectral decomposition is given by $\rho = \lambda \ketbra{\psi}{\psi} + (1-\lambda)(\textbf{1} - \ketbra{\psi}{\psi})$. Then by functional calculus, we have
\begin{align*}
    \log \rho = \log \frac{\lambda}{1 - \lambda} \ketbra{\psi}{\psi} + \log(1-\lambda) \cdot\textbf{1}.
\end{align*}
Similarly, using the definition of depolarizing channel, we have 
\begin{align*}
    \log(\Phi_{p,d}(\rho)) = \log \left(\frac{(1-p)\lambda + p/2}{(1-p)(1-\lambda) + p/2}\right) \ketbra{\psi}{\psi} + \log( (1-p)(1-\lambda) + p/2) \textbf{1}.
\end{align*}
Taking the Lipschitz semi-norm, we have 
\begin{align*}
   & |||\log(\Phi_{p,d}(\rho))|||_L = \left| \log \left(\frac{(1-p)\lambda + p/2}{(1-p)(1-\lambda) + p/2}\right) \right| |||\ketbra{\psi}{\psi}|||_L, \\
   & |||\log \rho|||_L = \left| \log \frac{\lambda}{1 - \lambda} \right| |||\ketbra{\psi}{\psi}|||_L.
\end{align*}
Denote $x = \lambda - 1/2 \in (-1/2, 1/2)$ and 
\begin{equation}
    f_p(x):= \left| \frac{\log (1+2(1-p)x) - \log (1-2(1-p)x)}{\log(1+2x) - \log(1-2x)} \right|
\end{equation}
Via elementary calculus, we see that $f_p(x)$ is an even function and it is decreasing on $(0,1/2)$. Therefore,
\begin{equation}
   \sup_{x \in (-1/2,1/2)} f_p(x) = \lim_{x \to 0}f_p(x) = 1-p.
\end{equation}
This implies that \begin{align*}
    |||\log(\Phi_{p,d}(\rho))|||_L = f_p(x) |||\log \rho|||_L \le (1-p)|||\log \rho|||_L,
\end{align*}
which concludes the proof.
\end{proof}
Using \eqref{eqn:upper 1 depolarizing}, \eqref{eqn:upper 2 depolarizing} and the above lemma, we have
\begin{prop}
    $\eta^D(\Phi_{p,2},\textbf{1}/2) \le (1-p)^2$.
\end{prop}

\begin{remark}
In the qubit setting, using the Pauli representation of a Hermitian operator $X = v_0\textbf{1} + v_1 \sigma_x + v_2 \sigma_y + v_3 \sigma_z$, and choosing $\mc S$ as the set of Pauli operators, we have
    \begin{equation}
    |||X|||_{\mc S} = 2\max\left\{ \sqrt{v_1^2 + v_2^2}, \sqrt{v_2^2 + v_3^2}, \sqrt{v_1^2 + v_3^2} \right\}.
\end{equation}
Via similar calculations as in \cite{belzig2024}, one can get explicit upper bound for Pauli-type channels, with the Lipschitz semi-norm given by commutator of Pauli operators. It recovers the sharp entropy contraction coefficients established in \cite{hiai2016contraction}. 
\end{remark}


\subsection{Mixing time estimate from Lipschitz contraction} \label{sec:mixing time estimate}
\subsubsection{Mixing time estimates from entropy contraction}
In this subsection, we use the estimates in the previous subsection to get explicit mixing time estimates for general ergodic unital channels, defined by 
\begin{equation}
    t_{\mathrm{mix}}(\varepsilon,\Phi):= \inf\{n\ge 1: \|\Phi^n(\rho) - \textbf{1}/d\|_1 \le \varepsilon,\quad \forall \rho \in S(\mc N)\}.
\end{equation}
An elementary inequality we will frequently use is the following Pinsker inequality and its reverse version, see \cite{vershynina2021quasi} for a summary:
\begin{equation}\label{Pinsker}
    \frac{1}{2}\|\rho - \sigma\|^2_1 \le D(\rho \| \sigma) \le \frac{\lambda_{\max}(\rho)}{\lambda_{\min}(\sigma)} \|\rho - \sigma\|_1.
\end{equation}
Since $\Phi$ is unital, $\sigma = \textbf{1}/d$ and we have $\Phi^{*,\mathrm{BKM}} = \Phi$. Applying Theorem \ref{main: entropy contraction upper}, we have 
\begin{theorem}\label{main: mixing time upper}
    Suppose $\Phi$ is unital with a unique fixed point $\textbf{1}/d$. Then there exists a universal constant $c_{abs}>0$ such that for any commutator semi-norm $|||\cdot|||_L$ with $\Lip_L(\Phi) < 1,\  \Lip_L(\Phi^*) \le  1$, 
    \begin{equation}
          t_{\mathrm{mix}}(\varepsilon,\Phi) \le c_{abs} \frac{\log(1/\varepsilon)+\log d }{-\log (\Lip_L(\Phi)) - \log( \Lip_L(\Phi^*))}.
    \end{equation}
\end{theorem}
\begin{proof}
    Denote the ``lazy" version of $\Phi$ as 
    \begin{align*}
        \wt \Phi = \frac{1}{2}\Phi + \frac{1}{2} \mc R_{\textbf{1}/d},\quad \mc R_{\textbf{1}/d}(X):= \tr(X) \textbf{1}/d. 
    \end{align*}
    Then $\wt \Phi$ is strictly positive and for any $N \ge 1$ and $\rho \in S(\mc N)$, 
    \begin{align*}
        \|\wt \Phi^N(\rho) - \textbf{1}/d\|_1 & \le \sqrt{2 D(\wt \Phi^N(\rho)\|\textbf{1}/d)} \\
        & \le \sqrt{2 \Lip_L\big((\wt \Phi^*)^N\big)} \sqrt{\max\left\{\Lip_L(\wt \Phi^N), \mathrm{Log}\Lip_L(\wt \Phi^N,\textbf{1}/d) \right\}} \\
        & \le \sqrt{2 \Lip_L\big((\wt \Phi^*)^N\big)\Lip_L(\wt \Phi^N)} \sqrt{c_{abs} \frac{d\log^2 d}{1-\frac{1}{2^N}}},
    \end{align*}
    where the first inequality follows from Pinsker inequality~\eqref{Pinsker}, the second inequality follows from Theorem \ref{main: entropy contraction upper} applied to $\wt \Phi^N$, and the last inequality follows from Proposition~\ref{prop:upper bound on log Lipschitz} with the fact $\lambda_{\min}(\wt \Phi^N) \ge (1-\frac{1}{2^N}) \frac{1}{d}$. Note that 
    \begin{align*}
        \Lip_L(\mc R_{\textbf{1}/d}) = \Lip_L(\mc R^*_{\textbf{1}/d}) = 0,
    \end{align*}
    and \begin{align*}
        \wt \Phi^N = \frac{1}{2^N}\Phi^N + (1- \frac{1}{2^N})\mc R_{\textbf{1}/d}. 
    \end{align*}
    Via convexity of Lipschitz constant, one has 
    \begin{align*}
        \Lip_L\big((\wt \Phi^*)^N\big)= \Lip_L\big(\frac{1}{2^N}(\Phi^*)^N + (1- \frac{1}{2^N})\mc R_{\textbf{1}/d} \big) \le \frac{1}{2^N}\Lip_L(\Phi^*)^N, 
    \end{align*}
    and similarly $\Lip_L\big(\wt \Phi^N\big) \le \frac{1}{2^N}\Lip(\Phi)^N$. Therefore, we get
    \begin{align*}
        \frac{1}{2^N}\|\Phi^N(\rho) - \textbf{1}/d\|_1 & = \|\wt \Phi^N(\rho) - \textbf{1}/d\|_1 \\
        & \le \sqrt{2 \Lip_L\big((\wt \Phi^*)^N\big)\Lip_L(\wt \Phi^N)} \sqrt{c_{abs} \frac{d\log^2 d}{1-\frac{1}{2^N}}} \\
        & \le \sqrt{2 \frac{1}{2^N}\Lip_L(\Phi^*)^N\frac{1}{2^N}\Lip_L(\Phi)^N} \sqrt{c_{abs} \frac{d\log^2 d}{1-\frac{1}{2^N}}},
    \end{align*}
    which produces 
    \begin{align*}
        \|\Phi^N(\rho) - \textbf{1}/d\|_1 \lesssim \sqrt{d}\log d \left(\sqrt{\Lip_L(\Phi) \Lip_L(\Phi^*)}\right)^N.
    \end{align*}
    Setting the right hand side to be less than $\varepsilon$, we get an upper bound on the mixing time
    \begin{equation}
         t_{\mathrm{mix}}(\varepsilon,\Phi) \le c_{abs} \frac{\log(1/\varepsilon)+\log d }{-\log (\Lip_L(\Phi)) - \log( \Lip_L(\Phi^*))}.
    \end{equation}
\end{proof}
We remark that using $\chi^2$-divergence, \cite{george2025} also got mixing time estimates recently.
 
\subsubsection{Estimates of cost induced mixing time}
We now derive an upper bound on the cost induced mixing time from the Lipschitz constant. The following lemma, which connects transportation cost with the Lipschitz constant of a given quantum channel $\Phi$, enabling us to derive an upper bound on the cost induced mixing time defined by
\begin{equation}\label{def:L induced mixing time}
    t_{\mathrm{mix}}^L(\varepsilon,\Phi^*)= \inf\{n \ge 1: \Cost_L^{cb}((\Phi^*)^n)\ge (1-\varepsilon) \Cost_L^{cb}(E_\fix)\}. 
\end{equation}
Here we assume $E_{\fix}$ is the conditional expectation onto the fixed point algebra: $\Phi^* \circ E_\fix = E_\fix$. 
\begin{lemma}\label{lemma:connection 1}
Under the additional assumption that
\[
\Phi^* \circ E_{\fix} = E_{\fix} \circ \Phi^* = E_{\fix},
\]
we have
\[
\Cost_L^{cb}(\Phi^*) \ge \bigl(1 - \Lip_L^{cb}(\Phi^*)\bigr)\, \Cost_L^{cb}(E_\fix).
\]
\end{lemma}
\begin{proof} 
Given any $\varepsilon > 0$, there exists $n \ge 1$ and $X_n \in \mb M_n \otimes \mc N$, such that $|||X_n|||_{L^{(n)}} \le 1$ and 
\begin{equation}
    \|(id_{\mb M_n}\otimes E_{\fix})(X_n) - X_n\|_{op} \ge \Cost_L^{cb}(E_\fix) - \varepsilon.
\end{equation}
Then 
\begin{align*}
    & \|(id_{\mb M_n}\otimes \Phi^*)(X_n) - X_n\|_{op} \\
    & \ge \|(id_{\mb M_n}\otimes E_{\fix})(X_n) - X_n\|_{op} - \|(id_{\mb M_n}\otimes \Phi^*)(X_n) - (id_{\mb M_n}\otimes E_{\fix})(X_n)\|_{op} \\
    & = \|(id_{\mb M_n}\otimes E_{\fix})(X_n) - X_n\|_{op} - \|(id_{\mb M_n}\otimes \Phi^*)(X_n) - (id_{\mb M_n}\otimes E_{\fix} \circ \Phi^*)(X_n)\|_{op} \\
    & \ge \Cost_L^{cb}(E_\fix) - \varepsilon - \|(id_{\mb M_n}\otimes id_{\mc B(\mc H)} - id_{\mb M_n}\otimes E_{\fix})\left((id_{\mb M_n}\otimes \Phi^*)(X_n) \right)\|_{op} \\
    & \ge \Cost_L^{cb}(E_\fix) - \varepsilon - \Cost_L^{cb}(E_\fix)|||(id_{\mb M_n}\otimes \Phi^*)(X_n)|||_{L^{(n)}} \\
    & \ge \Cost_L^{cb}(E_\fix) - \varepsilon - \Cost_L^{cb}(E_\fix) \Lip_L^{cb}(\Phi^*) \\
    & = (1-\Lip_S^{cb}(\Phi))\Cost_L^{cb}(E_\fix) - \varepsilon. 
\end{align*}
For the first inequality, we use the triangle inequality; for the first equality, we use the assumption $E_{\fix} \circ \Phi^*= E_{\fix}$; for the third inequality, we use the definition of $\Cost_L^{cb}(E_\fix)$; for the fourth inequality, we use the definition of $\Lip_L^{cb}(\Phi^*)$. Since $\varepsilon$ is arbitrary, we conclude the proof of $\Cost_L^{cb}(\Phi^*) \ge \bigl(1 - \Lip_L^{cb}(\Phi)\bigr)\, \Cost_L^{cb}(E_\fix)$.
\end{proof}
The following result shows that $\Lip_L^{cb}(\Phi) < 1$ implies an upper bound on $t_{\mathrm{mix}}^L(\varepsilon,\Phi^*)$:
\begin{prop}\label{prop:cost induced mixing time upper}
    Suppose $\Phi^* \circ E_{\fix} = E_{\fix} \circ \Phi^* = E_{\fix}$ and $\Lip_L^{cb}(\Phi) < 1$, then for any $\varepsilon > 0$, we have
    \begin{equation}
       t_{\mathrm{mix}}^L(\varepsilon,\Phi^*) \le \frac{\log (1/\varepsilon)}{ -\log (\Lip_L^{cb}(\Phi^*))}.
    \end{equation}
\end{prop}
\begin{proof}
    Via Lemma \ref{lemma:connection 1}, for any $n\ge 1$, via submultiplicativity of Lipschitz constant,
\begin{align*}
    \Cost_L^{cb}((\Phi^*)^n) & \ge \bigl(1 - \Lip_L^{cb}((\Phi^*)^n)\bigr)\, \Cost_L^{cb}(E_\fix) \\
    & \ge \bigl(1 - \Lip_L^{cb}(\Phi^*)^n\bigr)\, \Cost_L^{cb}(E_\fix).
\end{align*}
Recall the definition of $t_{\mathrm{mix}}^L(\varepsilon,\Phi^*) $ in~\eqref{def:L induced mixing time}, for any $\varepsilon>0$, if $n \ge \frac{\log \varepsilon}{ \log (\Lip_L^{cb}(\Phi))}$, then we have 
\begin{align*}
    \Cost_L^{cb}(\Phi^*\circ \cdots \circ \Phi^*) \ge (1-\varepsilon) \Cost_L^{cb}(E_\fix),
\end{align*}
which concludes the proof of the upper bound.
\end{proof}
\begin{remark}
     If $\Phi$ admits only one fixed point state $\sigma$ (called primitive), with the choice of $E_{\fix}(X) = \tr(\sigma X) \textbf{1}$, we always have $\Phi^* \circ E_{\fix} = E_{\fix} \circ \Phi^*= E_{\fix}$.
\end{remark}
\medskip

\noindent \textbf{Acknowledgements}. The authors thank Li Gao for helpful discussions and sharing some his ideas. The authors thank the Institute for Pure and Applied Mathematics (IPAM) for its hospitality in hosting them as long-term visitors during the semester-long program “Non-commutative Optimal Transport” in Spring 2025. 
\medskip

\noindent \textbf{Funding}. MJ's research is partially supported by NSF Grant DMS-2247114. PW's research is partially supported by Canada First Research Excellence Fund (CFREF). Part of this research was performed while the author was visiting the Institute for Pure and Applied Mathematics (IPAM), which is supported by the National Science Foundation (Grant No. DMS-1925919).
\medskip

\noindent \textbf{Data Availability}. Data sharing not applicable to this article as no datasets were generated or analyzed during
the current study. 
\medskip

\noindent \textbf{Conflict of interest}. The authors have no conflicts of interest to declare that are relevant to the content of this article.

\bibliographystyle{marcotomPB} 
\bibliography{cost}
\end{document}